\documentclass[aip,jmp,amsmath,amssymb,reprint,onecolumn,nofootinbib,notitlepage]{revtex4-1}

\usepackage{graphicx}
\usepackage{bm}
\usepackage[utf8]{inputenc}
\usepackage[T1]{fontenc}

\makeatletter
\def\@email#1#2{\endgroup
 \patchcmd{\titleblock@produce}
  {\frontmatter@RRAPformat}
  {\frontmatter@RRAPformat{\produce@RRAP{*#1\href{mailto:#2}{#2}}}\frontmatter@RRAPformat}
  {}{}
}
\makeatother

\usepackage{amsthm}
\usepackage{physics}
\usepackage{enumerate}
\usepackage[shortlabels]{enumitem}

\PassOptionsToPackage{hyphens}{url} 
\usepackage[breaklinks=true]{hyperref}
\bibpunct{[}{]}{;}{n}{}{} 
\usepackage[capitalise]{cleveref}

\newtheorem{theorem}{Theorem}
\newtheorem{proposition}[theorem]{Proposition}

\newtheorem{corollary}[theorem]{Corollary}
\newtheorem{lemma}[theorem]{Lemma}

\theoremstyle{definition}
\newtheorem{definition}[theorem]{Definition}
\newtheorem{assumption}{Assumption}
\theoremstyle{remark}
\newtheorem{remark}{Remark}

\crefname{enumi}{}{}

\DeclareMathOperator*{\argmin}{\mathrm{argmin}}

\DeclareMathOperator{\dom}{dom}
\newcommand{\Hperhomdual}{H^{-1}_\mathrm{per,hom}}
\newcommand{\Hperhom}{H^{1}_\mathrm{per,hom}}
\newcommand{\Xdens}{\mathcal{X}}
\newcommand{\Xdensaff}{\mathcal{X}_\mathrm{aff}}
\newcommand{\Xpot}{\mathcal{X}^*}
\newcommand{\GG}{\mathbf{G}}
\newcommand{\RR}{\mathbf{R}}
\newcommand{\densSet}{\mathcal{I}_N(\Omega)}
\newcommand{\weakto}{\rightharpoonup}

\newcommand{\FLL}{F_\mathrm{LL}}
\newcommand{\eps}{\varepsilon}
\newcommand{\rr}{\mathbf{r}}
\renewcommand{\Tr}{\operatorname{Tr}\,}
\newcommand{\mpgrid}{\mathcal{G}_\mathrm{MP}}
\newcommand{\rme}{\mathrm{e}}
\newcommand{\rmi}{\mathrm{i}}
\newcommand{\rmr}{\mathrm{d}}
\mathchardef\mhyphen="2D 
\newcommand \wlim {\mathop\mathrm{w\mhyphen lim}}

\begin{document}

\title{Moreau--Yosida-based Kohn--Sham Inversion for Periodic Systems}
\author{Vebj{\o}rn H. Bakkestuen}
\email{vebjorn.bakkestuen@oslomet.no}
\affiliation{\mbox{Department of Computer Science, Oslo Metropolitan University, Oslo, Norway}}

\author{Michael F. Herbst}
\affiliation{Mathematics for Materials Modelling (MatMat), Institute of Mathematics \& Institute of Materials, École Polytechnique Fédérale de Lausanne, Lausanne, Switzerland}

\author{Vegard Falm{\aa}r}
\affiliation{\mbox{Department of Computer Science, Oslo Metropolitan University, Oslo, Norway}}

\author{Markus Penz}
\affiliation{\mbox{Department of Computer Science, Oslo Metropolitan University, Oslo, Norway}}
\affiliation{Arnold Sommerfeld Centre for Theoretical Physics, Ludwig-Maximilians-Universität München, Munich, Germany}
\affiliation{Max Planck Institute for the Structure and Dynamics of Matter and Center for Free-Electron Laser Science, Hamburg,
Germany}

\author{Andre Laestadius}
\email{andre.laestadius@oslomet.no}
\affiliation{\mbox{Department of Computer Science, Oslo Metropolitan University, Oslo, Norway}}
\affiliation{Hylleraas Centre for Quantum Molecular Sciences, Department of Chemistry, University of Oslo, Oslo, Norway}

\date{\today}

\begin{abstract}
    Density-potential inversion for periodic systems within Moreau--Yosida-regularised density-functional theory is investigated, both theoretically and numerically. We develop the framework in a periodic homogeneous Sobolev space and use it to recover the exchange-correlation potential of Kohn--Sham theory through a limiting procedure. A key analytical ingredient is the proof of lower semicontinuity of the non-interacting kinetic-energy functional in the chosen topology. The proximal mapping, together with its algorithmic evaluation, plays a central role in the resulting inversion scheme. Numerical experiments illustrate the performance and properties of the method for both the Kohn--Sham and Gross--Pitaevskii equations.
\end{abstract}


\maketitle

\section{Introduction}
Density-functional theory (DFT) calculations are an indispensable tool in theoretical and practical applications of the many-body electronic Schrödinger equation across disciplines such as chemistry, materials science, and solid-state physics~\cite{Burke2012,Verma2020,Teale2022}.
Its central variable is the one-particle density, here denoted $\rho(\rr)$, $\rr \in \Omega \subseteq \mathbb R^d$, where for the standard formulation of DFT $\Omega=\mathbb{R}^3$.   
Practical DFT calculations are performed almost exclusively within the Kohn--Sham (KS) framework, in which the sought-after ground-state density $\rho_\mathrm{gs}(\rr)$ of an interacting system is assumed to be a solution to a non-interacting problem as well. This fictitious KS system is determined by the effective KS potential $v_\mathrm{KS}$, a local potential responsible for generating the \emph{same} ground-state density of the original, interacting system.

The universal density functional $F(\rho)$ encapsulates all many-body effects such that the ground-state energy for a given external scalar potential $v: \Omega \to \mathbb R$  can be written as
\begin{equation*}
    E(v) =  F(\rho_\mathrm{gs}) + \langle v, \rho_\mathrm{gs} \rangle, \qquad   \langle v, \rho \rangle = \int_\Omega v(\rr) \rho(\rr) \mathrm{d} \rr .
\end{equation*}

The explicit form of $F$ is unknown in general.  
Applications of Kohn--Sham density-functional theory (KS-DFT) rely on approximations to the functional from which the effective KS potential is derived. Over the years, substantial efforts have been devoted to obtaining accurate approximations~\cite{Toulouse2023} which typically address the \textit{exchange-correlation} (xc) energy $E_\mathrm{xc}$ in the decomposition
\begin{equation}\label{eq:DensityFunctionalDecomp}
    F(\rho) = T(\rho) + E_\mathrm{H}(\rho) + E_\mathrm{xc}(\rho).
\end{equation}
Here $T(\rho)$ is the non-interacting kinetic-energy functional and $E_\mathrm{H}(\rho)$ is the (direct) Hartree energy.
Given the potential of the physical system $v$, 
the KS potential is typically split into $v_\mathrm{KS} =v + v_\mathrm{H} + v_\mathrm{xc}$, where $v_\mathrm{H}$ is the Hartree (electrostatic) potential and $v_\mathrm{xc}$ is the xc potential that must be approximated.
Although DFT in general and KS-DFT in particular are exact reformulations of the many-body Schrödinger equation and significant advancements have been made over the years, KS-DFT with the current functional approximations remains unable to predict certain physical processes.
Inaccurate approximations to $v_\mathrm{xc}$ are important sources of errors in KS-DFT calculations, e.g., band-gap underestimation~\cite{Sham1983PRL,Sham1985PRB,Perdew1982PRL,Perdew1983PRL,Gruning2006} and self-interaction errors~\cite{ZhangYang-jcp1998}.
Developing improved functional approximations continues to be a central research focus~\cite{Mardirossian2017}. A key difficulty lies in the limited mathematical insight into the relationship between the exact universal density functional and the widely used approximations~\cite{Teale2022}, which hampers the systematic design of approximate functionals.

To obtain an effective KS potential corresponding to a given interacting ground-state density is a highly non-trivial task~\cite{Gar21,Trushin2024}. We aim to shed more light on this issue by both mathematically and numerically addressing this inverse KS problem: given knowledge of the exact (or near-exact) ground-state density, reconstruct the effective KS potential. 
Despite its importance, both as a tool to aid the development of better functional approximations and for the fundamental understanding of KS-DFT, KS inversion remains a challenging problem~\cite{Shi2021,Aryasetiawan1988,Knorr-Godby-1992,Gorling1992,ZP1993,Wang1993,ZMP1994,Knorr_1994,vanLeeuwen_1994,Yang_Wu_2002,Wu-Yang-2003,Peirs2003,Kadantsev2004,Bulat2007,Gaiduk2013,Wagner2014,JensenWasserman2017,Zhang2018,Ou2018,kumar2019universal,Kanungo2019,Kumar2020,Garrick2020,Callow2020,Shi2021,Erhard2022,Gould2023,Ravindran2025,Herbst_2025,Penz2025-imag-time}. Several explicit inversion schemes have been proposed, where we highlight the pioneering work of \citet{vanLeeuwen_1994} and those of \citet{ZMP1994} (ZMP) and \citet{Wu-Yang-2003}. Earlier works focused on molecules and isolated systems, while only recently solid-state problems have been investigated~\cite{Aouina_2023,Ravindran2025,Herbst_2025,Bohle2026}.
Notwithstanding these recent developments, a robust and efficient numerical scheme for KS inversion largely remains an open problem~\cite{Shi2021,Nam2021,Shi2022,Crisostomo2023,Wrighton2023,Gould2023}.

The effective (or just the xc) potential was derived as a strict mathematical limit within the Moreau--Yosida (MY) regularised formulation of DFT~\cite{Kvaal2014,KSpaper2018,Laestadius_2019_CDFT,KS_PRL_2019,PRLerrata} in \cite{Penz_2023}, recently summarised in \cite{Penz_2025Perspective_arXiv}. 
Evaluating this limit involves obtaining the proximal point of a ground-state density,
which is the minimiser of the infimal convolution defining the MY envelope.
In addition, one needs to apply the duality mapping $J: X \to X^*$ between the density and potential spaces. 
In this set-up, the MY approach regularises $T(\rho)$ (or some variant of it), yielding a smooth functional $T^\eps(\rho)$ whose proximal mapping produces a proximal density $\rho^\eps$. The effective potential $v_\mathrm{KS}$ then emerges as
\begin{equation*}
   v_\mathrm{KS} = \lim_{\eps \to 0^+} \frac{1}{\eps} J(\rho^\eps - \rho_\mathrm{gs}),
\end{equation*}
under the assumption that a KS potential indeed exists for the given ground-state density $\rho_\mathrm{gs}$ (i.e., $-v_\mathrm{KS}\in \underline{\partial} T(\rho_\mathrm{gs})\neq \emptyset$).
This connection was recently explored for three-dimensional periodic systems~\cite{Herbst_2025}, and 
has also been analysed in a one-dimensional periodic setting~\cite{Bohle2026}.
Moreau--Yosida regularisation provides a general framework for KS inversion suited for periodic settings which will be explored in this work.

\subsection{Main contributions} 
In this work we develop a mathematically rigorous framework for density-potential inversion in periodic systems and demonstrate its theoretical and practical relevance for KS-DFT. Our main contributions are the following.
\begin{enumerate}[label=(\roman*)]
    \item \emph{A general inversion scheme.} We formulate a Moreau--Yosida-based inversion algorithm for densities in reflexive Banach spaces, while for some of the results a Hilbert-space setting is necessary. 
    The scheme centres on the proximal mapping  and requires only the technical conditions listed in
    Assumption~\ref{assump:FX}, notably the lower‑semicontinuity of the considered (non-interacting) density functionals.
    In addition, we can show that the propagation of errors in the input density is bounded utilising the non-expansiveness properties of the proximal mapping.
    
    \item \emph{Application to periodic homogeneous Sobolev spaces.} In particular, we apply our general inversion framework to the dual of the periodic homogeneous Sobolev space, 
    $H_\mathrm{per,hom}^{-1}(\Omega)$, that is a Hilbert space. We prove the lower‑semicontinuity of the kinetic-energy constrained-search functional 
    for this particular topology, which guarantees the applicability of our general Moreau--Yosida inversion scheme. In this setting, the framework offers an algorithm to determine the exact exchange‑correlation (xc) potential of KS-DFT. Amongst other things, this places the ZMP method on firm functional-analytic footing for the first time.
    
    \item \emph{Numerical analysis of the inversion scheme.}
    The theoretical results are investigated 
    with numerical examples using the periodic one-dimensional Gross--Pitaevskii equation and the Kohn--Sham equation for three-dimensional bulk materials. 
    In particular, we emphasise the performance of different algorithms for determining the proximal density, and the specific form of the non-interacting functional. 
\end{enumerate}

\subsection{Outline}
The remainder of the article is structured as follows. 
We begin with some preliminaries in \cref{sec:Prel}. 
\Cref{subsec:DFT-KS} reviews the essential mathematical content of density-functional theory and the Kohn--Sham formulation, emphasising the distinction between interacting and non-interacting systems. \Cref{subsec:BanachPrelims} gathers the functional-analytic preliminaries required for our analysis. In \Cref{subsec:MYreg}, we present the Moreau--Yosida regularisation for the purpose of applying it to density-potential (KS) inversion. This includes the introduction of the proximal mapping and establishing its basic properties.
\Cref{subsec:homogeneousSobolev} specialises the abstract framework to periodic systems by introducing the periodic homogeneous Sobolev spaces, and certain connections to the density-functional problems are proven.
Next, \Cref{sec:InversionScheme} presents the main results of the Moreau--Yosida-based density-potential inversion scheme. We first, in \cref{subsec:InversionAlgorithm}, present and prove results for the general inversion algorithm under Assumption~\ref{assump:FX}. 
Additionally, we investigate the impact of different guiding functionals and the choice of function spaces. In \cref{subsec:kin-functional} we prove that the kinetic-energy constrained-search functional is lower semicontinuous on the space presented in~\Cref{subsec:homogeneousSobolev} which shows that the technical assumptions listed in Assumption~\ref{assump:FX} are fulfilled. This connects the functional-analytic machinery to the density-functional problem of solving a Schrödinger-type equation.
\cref{subsec:PeriodicIKS} presents the Moreau--Yosida-based Kohn--Sham inversion scheme where we also remark on the exact connection to the ZMP method. \Cref{sec:Numerics} showcases numerical applications, where in \cref{subsec:InversionApproaches} we first outline the used algorithms that subsequently are applied to the periodic Gross--Pitaevskii equation (\cref{subsec:GPE}) and to bulk materials (\cref{subsec:BulkMaterials}).
We numerically examine the behaviour of the different algorithms for determining the proximal density, the form of the guiding functional, and how errors in the input density propagate.

\section{Preliminaries}\label{sec:Prel}
To set the theoretical stage for the inversion scheme, let us introduce the two main ingredients---density-functional theory and the Moreau--Yosida regularisation. The background of DFT is presented for general domains in arbitrary dimensions focusing on the essential ingredients for a rigorous inversion scheme. Moreover, the MY regularisation is presented on a general class of Banach spaces---a formulation we will employ to present the general KS inversion scheme. We further discuss some useful tools on Banach spaces and, in addition, introduce a specialised Hilbert-space setting in the form of periodic homogeneous Sobolev spaces that will be used for the formulation of the method in the later sections.

\subsection{Density-Functional Theory \& the Kohn--Sham Method}\label{subsec:DFT-KS}
The usual formulation of DFT considers a system of $N$ electrons subject to a scalar potential $v(\rr)$ on $\mathbb{R}^3$. In the present formulation, we allow for a more general setting in which the spatial domain $\Omega$ is either an open subset of $\mathbb{R}^d$ or the flat torus $\mathbb{T}^d$,  where $d \geq1$. The standard formulation of DFT is then recovered by choosing $\Omega= \mathbb{R}^3$. Moreover, let $\mathcal{H}_N$ denote the usual $N$-particle fermionic Hilbert space on $\Omega^N$, defined with the appropriate boundary conditions---for example, periodic boundary conditions for periodic systems. For simplicity, we will disregard spin in our treatment.

Now, let $v:\Omega \to \mathbb R$ be given and denote by $\hat{H}: \mathcal{H}_N \to \mathcal{H}_N$ the many-body electronic Hamiltonian  $\hat{H} = \hat{T} + \hat{W} + \hat{V}$ where $\hat{T}$ denotes the kinetic energy operator, $\hat{W}$ is the (typically Coulombic) two-body interaction operator, and $\hat{V} = \sum_{j=1}^N v(\rr_j)$ the potential energy operator, equalling $v(\rr)$ lifted to the $N$-body domain. Here, $\hat V$ and $\hat{W}$ are always assumed to be such that the KLMN theorem~\cite[Thm.~X.17]{reed-simon-2} guarantees a Hamiltonian that is bounded below and self-adjoint with form domain $Q(\hat H)=Q(\hat T)$ (see \cite[Sec.~VIII.6]{reed-simon-1}; since $\hat T = -\frac{1}{2}\Delta$ this corresponds to the set of all $\psi \in \mathcal{H}_N$ for which $\|\nabla\psi\|_{L^2(\Omega^N)}<\infty$). Later, when a space for potentials $v$ is defined, it must be such that every element still respects this condition.
The ground-state energy of the system is given by the variational problem 
\begin{equation}\label{eq:VariationalProblem}
	E(v) = \inf_{\psi \in \mathcal{W}_N} \langle \hat{H}\psi,  \psi \rangle = \inf_{\psi \in \mathcal{W}_N} \langle(\hat{T} + \hat{W} + \hat{V})\psi, \psi\rangle , 
    \quad \langle f,g\rangle = \int_{\Omega^N} \overline{f} g \dd{\rr_1}\cdots\dd{\rr_N}, 
\end{equation} 
where $\mathcal{W}_N \subset \mathcal{H}_N$ denotes the set of admissible states. This set consists of all normalised states in the form domain of $\hat{H}$,
\begin{equation*}
	\mathcal{W}_N = \qty{\psi \in Q(\hat H) : \norm{\psi}_{L^2(\Omega^N)} = 1}.
\end{equation*}

Rather than tackling the full many-body wave-function, DFT is formulated in terms of the ground-state density $\rho_\mathrm{gs}$. In general, a density $\rho$ can be computed from a state $\psi \in \mathcal{W}_N$ by integration of $|\psi|^2$ over all but one coordinate,
\begin{equation} \label{eq:def:rho}
	\rho_\psi(\rr) = N \int_{\Omega^{N-1}} \abs{\psi(\rr, \rr_2, \rr_3,\dots,\rr_N)}^2 \dd{\rr_2}\dd{\rr_3}\cdots\dd{\rr_N}.
\end{equation}
A density is said to be a ground-state density $\rho_\mathrm{gs}$ if it is computed from a ground state $\psi_\mathrm{gs}\in \mathcal{W}_N$ (i.e., a minimiser of $E(v)$ in \cref{eq:VariationalProblem} satisfying $\hat{H}\psi = E \psi$ in the weak sense). In general, if a density $\rho$ is computed from a state $\psi$ using \cref{eq:def:rho}, we write $\psi \mapsto \rho$. Following the notation of \citet{Lieb1983}, let $\densSet$ denote the set of $N$-representable densities, that is, all functions $\rho\in L^1_\#(\Omega)$ such that $\sqrt{\rho} \in H_\#^1(\Omega)$ (the subscript $\#$ signals the appropriate boundary conditions to be included) and $\int_\Omega\rho(\rr) \dd{\rr} = N$. It holds that $\rho \in \densSet$ if and only if there exist a state $\psi \in \mathcal{W}_N$ with $ \psi \mapsto \rho$~\cite{Lieb1983}.
The central object in DFT is the \textit{universal density functional} $F(\rho)$, which, by the fact that the energy $E(v)$ is a concave functional, can be defined as the Legendre--Fenchel conjugate of $E(v)$ (with an unconventional choice of signs)
\begin{equation} \label{eq:def:UniversalFunctional}
	F(\rho) = \sup_{v\in X^\ast} \qty{E(v) - \langle v, \rho \rangle }.
\end{equation}
Here, $X^*$ denotes the dual of the selected density space $X$, which we assume to be a real Banach space with $\densSet \subset X$. We use the notation $\langle v, \rho \rangle $ for the dual pairing between elements $\rho \in X$ and $v \in X^\ast$. By construction, $F(\rho)$ is a proper, convex, and lower semicontinuous (l.s.c.) functional. 
\Cref{eq:def:UniversalFunctional} is also known as the \emph{Lieb functional}, and by Legendre--Fenchel back-transform the energy can be again obtained as
\begin{equation*}
    E(v) = \inf_{\rho \in X} \qty{F(\rho) + \langle v, \rho \rangle}.
\end{equation*}
This retrieves \cref{eq:VariationalProblem} from \cref{eq:def:UniversalFunctional}. However, other definitions of universal density functionals exist in which the connection to the many-body state, and thus also to the Schrödinger equation, is more direct. In particular, let us consider the energy functional computed as the variation over density matrices $\Gamma$, i.e., mixed states. Analogously to the set of admissible pure states $\mathcal{W}_N$, we denote by $\mathcal{D}_N$ the set of admissible density matrices,
\begin{equation*}
    \mathcal{D}_N = \qty{\Gamma\in\mathcal{B}(\mathcal{H}_N) : \Gamma\geq 0, \, \Tr\Gamma = 1,\,  \Tr [\hat T\Gamma]<\infty }.
\end{equation*} 
Here $\mathcal{B}(\mathcal{H}_N)$ denotes the set of bounded linear operators on $\mathcal{H}_N$ and any $\Gamma \in \mathcal{D}_N$ is a positive semi-definite, trace-class operator normalised to unit trace with finite kinetic energy. It then follows that 
\begin{equation} \label{eq:EnergyFDM}
    E(v) = \inf_{\Gamma\in \mathcal{D}_N} \Tr[\hat{H}\Gamma] 
    = \inf_{\rho\in X} \qty{\inf_{\Gamma\mapsto \rho} \Tr[(\hat{T} + \hat{W})\Gamma] + \Tr[\hat{V}\Gamma] } 
    = \inf_{\rho\in X} \qty{F_\mathrm{DM}(\rho) + \langle v, \rho \rangle } ,
\end{equation}
where we have introduced the \textit{mixed-state constrained-search functional} $F_\mathrm{DM} : X \to \overline{\mathbb{R}} = \mathbb{R}\cup \{+\infty\} $ as 
\begin{equation*}
    F_\mathrm{DM}(\rho) = \inf_{\mathcal{D}_N \ni \Gamma\mapsto \rho} \Tr[(\hat{T} + \hat{W})\Gamma].
\end{equation*}
By varying only over pure states $\Gamma = \langle\psi,\cdot\rangle\psi$ one instead obtains the \textit{pure-state constrained-search functional} or Levy--Lieb functional $\FLL: X \to \overline{\mathbb{R}}$ as
\begin{equation*}
    \FLL(\rho) = \inf_{\mathcal{W}_N \ni \psi \mapsto \rho} \langle (\hat{T} + \hat{W})\psi,  \psi \rangle.
\end{equation*}
Since $\FLL(\rho)$ is obtained by the minimisation over the restricted set of pure states, it immediately follows that $F_\mathrm{DM}(\rho) \leq \FLL(\rho)$ for all $\rho \in X$.
By construction, all three universal functionals give the same energy functional as 
\begin{equation*}
    E(v) = \inf_{\rho \in X} \qty{F(\rho) + \langle v, \rho \rangle}
    = \inf_{\rho \in X} \qty{F_\mathrm{DM} (\rho) + \langle v, \rho \rangle}
    = \inf_{\rho \in X} \qty{\FLL(\rho) + \langle v, \rho \rangle},
\end{equation*}
but the back transformation only holds if the universal functional is convex and l.s.c.~\cite[Th.~2.22]{Barbu-Precepanu}. This can be shown to hold for $F_\mathrm{DM}(\rho)$ on $L^1(\mathbb{R}^3)\cap L^3(\mathbb{R}^3)$ and thus $F(\rho) = F_\mathrm{DM}(\rho)$ in this case, while $F(\rho)$ is still the closed convex envelope of $\FLL(\rho)$~\cite{Lieb1983}.
This identity of the Lieb functional with the mixed-state constrained-search functional is an integral part of the convex-analytic formulation of DFT, and it cannot be taken for granted in every functional-analytic setting~\cite{Penz_2025Perspective_arXiv}. Only if $F(\rho) = F_\mathrm{DM}(\rho)$, obtaining a \emph{ground-state} density $\rho$ is actually equivalent to saturating the Fenchel--Young inequality $E(v) \leq F(\rho) + \langle v, \rho \rangle$ for a given $v$,
\begin{equation*}
	E(v) = F(\rho) + \langle v, \rho \rangle.
\end{equation*}
If, as remarked by \citet{Sutter_2024JPhysA}, the density space $X$ has a topology \emph{finer} than that of $L^1(\mathbb{R}^3)\cap L^3(\mathbb{R}^3)$, the set of lower semicontinuous functions gets larger and the statement $F(\rho) = F_\mathrm{DM}(\rho)$ still holds. However, in \cref{subsec:homogeneousSobolev} we will instead consider a space with coarser topology and thus need to prove lower semicontinuity of the functional separately (see \cref{subsec:kin-functional}).

In the KS formulation of DFT, one relies on the assumption that, for an interacting $N$-particle system subject to the external potential $v$ with a ground-state density $\rho$, there exist a \textit{non-interacting} system that reproduces the same ground-state density. The KS assumption is therefore that of the existence of a fictitious system described by the Hamiltonian
\begin{equation*}
	\hat{H}_\mathrm{KS} = \hat T +  \sum_{j=1}^N v_\mathrm{KS}(\rr_j) 
    = \sum_{j=1}^N  \qty( \hat{t}_{j} + v (\rr_j) + v_{\mathrm{H}}(\rr_j) + v_{\mathrm{xc}}(\rr_j))
\end{equation*}
where we denote by $\hat t_j = -\frac{1}{2}\nabla_j^2$ the one-particle kinetic-energy operator. Analogously to $F(\rho)$ from the interacting formulation above, we denote by $T: X \to \overline{\mathbb{R}}$ the Legendre--Fenchel conjugate of the energy functional using $\hat T$ instead of $\hat H$ in Eq.~\eqref{eq:EnergyFDM}, and by $T_\mathrm{DM}: X \to \overline{\mathbb{R}}$ and $T_\mathrm{LL}: X \to \overline{\mathbb{R}}$ the corresponding mixed-state and pure-state (kinetic-only) constrained-search functionals, respectively.

Within the study of DFT, and in particular KS-DFT, the $v$-representability problem is a central issue. It amounts to the fact that in the usual setting not every $N$-representable density $\rho\in\densSet$ can be computed from the solution of a Schrödinger equation with some potential $v\in X^*$~\cite{Englisch_1983}.
Therefore it would be of great interest to classify all $v$-representable densities that can be calculated from ground states for a given $v\in X^*$.
Although some progress has been made on the characterisation of $v$-representable sets~\cite{LaestadiusKS2014}, the problem was solved only on bounded one-dimensional domains \cite{Sutter_2024JPhysA,Corso2026}, and one thus usually relies on an assumption of $v$-representability for all considered densities. In order for the KS approach of DFT to work, one even needs to assume joint interacting and non-interacting $v$-representability. That is, one assumes that for a given density $\rho$ there exists a potential $v_\mathrm{ext}$ such that $\rho$ stems from the solution of an interacting Schrödinger equation and simultaneously a potential $v_\mathrm{KS}$ such that $\rho$ also stems from the solution of a non-interacting Schrödinger equation.
This issue of $v$-representability is avoided by switching to the Moreau--Yosida-regularised version of the universal functional considered in \cref{subsec:MYreg}.

\subsection{Banach Space Preliminaries}\label{subsec:BanachPrelims}
The duality mapping is the canonical map from a Banach space $X$ to its dual space $X^\ast$, i.e., it maps an element of $X$ to elements in $X^\ast$.
We will for the purpose of the inversion scheme make use of the \textit{normalised duality mapping}.
\begin{definition}\label{def:DualityMapping}
	By $J : X  \rightrightarrows X^\ast$ we denote the \textit{duality mapping},
	\begin{equation*}
		J(\rho) = \qty{ u \in X^\ast : \langle u, \rho \rangle = \norm{\rho}_X^2 = \norm{u}_{X^\ast}^2	}.
	\end{equation*}
\end{definition}
Here $\langle u, \cdot \rangle = u(\cdot)$ denotes the dual pairing of elements in $X$ with an element $u \in X^\ast$, a bounded linear functional on $X$. Note that if the dual space $X^\ast$ is strictly convex, then the duality mapping is single-valued and demicontinuous, i.e., continuous from $X$ to the dual $X^\ast$ endowed with the weak-$\ast$ topology~\cite[Thm.~1.2]{Barbu_2010}. In particular, uniform convexity of $X^\ast$ implies the uniform continuity of $J$ on every bounded subset of $X$. Moreover, it follows that $J$ is homogeneous for every real Banach space, but additive (and by homogeneity then linear) if and only if $X$ is a Hilbert space~\cite[Prop.~1.117]{Barbu-Precepanu}. Also, for Hilbert spaces, the duality mapping corresponds to the Riesz map.

Moreover, when studying the exact universal density functional, a generalised notion of differentiability is necessary as the universal density functional has been shown to be non-differentiable almost everywhere in the $L^1$-topology~\cite{Lammert2007}. Since the universal functional is convex, an appropriate generalised notion of differentiation is the subdifferential. 

\begin{definition}\label{def:Subdifferential}
	Let $f:X \to \overline{\mathbb{R}}$ be a proper, convex, and lower semicontinuous functional. Then the \textit{subdifferential} of $f$ at the point $\rho \in X$ is the set 
	\begin{equation*}
		\underline{\partial} f(\rho) = \qty{  u\in X^* \mid \forall\rho'\in X: f(\rho') \geq f(\rho)+\langle u,\rho'-\rho \rangle }.
	\end{equation*}
	An element $u\in \underline{\partial} f(\rho) $ is called a \textit{subderivative} of $f$ at $\rho$.
\end{definition}
In other words, the subdifferential is the set of all bounded tangent functionals of $f$ at $\rho$ that lie below $f(\rho)$.
In cases where $f$ is differentiable it simply holds $\underline{\partial} f(\rho) = \{ \rmr f(\rho) \}$.  Note that we write $\rmr f(\rho)$ for both the Gâteaux derivative at $\rho\in X$ (where $\rmr f(\rho)$ is assumed linear and bounded) and the Fréchet derivative at $\rho\in X$ (that additionally provides a \emph{uniform} first-order approximation by $\langle \rmr f(\rho),h \rangle$ in all directions $h\in X$)~\cite{Barbu-Precepanu}, depending on the available regularity of $f$.
Also note that the subdifferential $\underline{\partial} f$ of a proper, convex, and lower semicontinuous functional is \emph{maximally monotone}~\cite{rockafellar1970,Ismatov2025}, i.e., for all $\rho,\rho'\in X$ and any $u\in\underline{\partial} f(\rho)$, $u'\in\underline{\partial} f(\rho')$ it holds
\begin{equation*}
    \langle u-u',\rho-\rho' \rangle \geq 0
\end{equation*}
and the graph of $\underline{\partial} f$ cannot be extended without losing this monotonicity.

Furthermore, note that the continuous functional $\rho \mapsto \norm{\rho}_X^2$ is strictly convex if and only if $X$ is strictly convex itself.  
Generally, for Banach spaces it holds that $\underline{\partial}\tfrac{1}{2}\norm{\rho}_X^2 =  J(\rho)$~\cite[Ex.~2.32]{Barbu-Precepanu}, and if $X$ is smooth ($X^*$ strictly convex), the subdifferential of $\norm{\cdot}^2_X$ is a singleton. In such a setting we thus have $J(\rho)=\rmr\tfrac{1}{2}\norm{\rho}_X^2$.

\subsection{Moreau--Yosida Regularisation}\label{subsec:MYreg}
A crucial ingredient for the inversion algorithm of \cref{sec:InversionScheme} is the Moreau--Yosida regularisation of the universal functional. Generally, we define the Moreau--Yosida regularisation as follows.

\begin{definition}\label{def:MYRegularisation}
	Let $f : X \to \overline{\mathbb{R}}$ be a proper, convex, and lower semicontinuous functional. For $\eps > 0$, the \textit{Moreau--Yosida regularisation} of $f$ is 
    \begin{equation*}
    	f^\eps(\rho) = \inf_{\sigma \in X} \qty{ f(\sigma) + \frac{1}{2\eps} \norm{\sigma-\rho}^2_X}.
    \end{equation*}
\end{definition}

Let us next introduce the minimiser of the infimum in the Moreau--Yosida regularisation, and give some of its properties.

\begin{proposition}\label{prop:UniqueProx}
	Let $X$ be a reflexive and strictly convex Banach space, and $f:X \to \overline{\mathbb{R}} $ a proper, convex, and lower semicontinuous functional. Then for every $\rho \in X$ and $\eps>0$ there exists a unique minimiser  $\rho^\star \in X$ of $ f(\sigma) + \frac{1}{2\eps} \norm{\sigma-\rho}^2_X$.
\end{proposition}
\begin{proof}
	Since $X$ is reflexive and the subdifferential of $f$ is maximally monotone, the minimiser condition $0\in \underline\partial f(\rho^\star)+\frac{1}{\eps}J(\rho^\star-\rho)$ gives at least one solution~\cite[Th.~1.141]{Barbu-Precepanu}.
    By addition of the parabola $ \norm{\sigma-\rho}^2_X$ that is strictly convex (since $X$ is), the whole argument $f(\sigma) + \frac{1}{2\eps} \norm{\sigma-\rho}^2_X$ is strictly convex in $\sigma$ and thus gives exactly one minimiser.
\end{proof}
The existence of the optimiser $\rho^\star$ motivates the following definition. 
\begin{definition}\label{def:ProximalOperator}
	Let $X$ be a reflexive and strictly convex Banach space and $f:X \to \overline{\mathbb{R}}$ be a proper, convex, and lower semicontinuous functional. The \emph{proximal mapping} $\Pi_f^\eps : X \to X$  is defined as 
	\begin{equation*}
		\Pi_f^\eps(\rho) = \argmin_{\sigma\in X} \qty{f(\sigma) + \frac{1}{2\eps} \norm{\sigma-\rho}^2_X}
	\end{equation*}
    and $\Pi_f^\eps(\rho)$ is referred to as the proximal point or proximal density.
\end{definition}
Note that the proximal mapping (sometimes also called a proximal operator) is in general highly non-linear as the solution to a non-linear optimisation problem. The proximal mapping is also denoted  $\mathrm{prox}_{\eps f} (\rho)$ in the literature and we will later introduce the shorthand notation $\rho^\eps = \Pi_f^\eps(\rho_\mathrm{gs})$ when no confusion can arise due to the choice of argument $\rho_\mathrm{gs} \in X$ and functional $f: X \to \overline{\mathbb{R}}$.
Next, we summarise several useful properties of the regularisation. 
The proofs can be found in \cite[Thm.~2.58]{Barbu-Precepanu} and its preceding discussion, as well as collected in \cite[Lem.~1]{Penz_2023}. Point \ref{item:MY-prox-inequality} is proven in \cite[Thm.~2.9]{Barbu_2010}.

\begin{theorem}\label{thrm:MYRegProperties}
    Let $X$ be a strictly convex and reflexive Banach space with a strictly convex dual space $X^\ast$ and let $f:X \to \overline{\mathbb{R}}$ be a proper, convex, and lower semicontinuous functional. Then the following properties hold true:
    \begin{enumerate}[(i)]
        \item for all $\eps>\delta >0$ and all $\rho \in X$ it holds $f^\eps(\rho) \leq f^\delta(\rho) \leq f(\rho)$,
        \item\label{item:limit-f-eps} for all $\rho \in X$ it holds $\lim_{\eps \to 0^+} f^\eps(\rho) = f(\rho)$ pointwise from below,
        \item\label{item:MY-prox-inequality} for all $\rho \in X$ it holds $f(\Pi^\varepsilon_f(\rho)) \leq f^\eps(\rho) \leq f(\rho)$,
        \item\label{item:limit-prox} for all $\rho \in \mathrm{dom}(f)$ it holds $\lim_{\eps \to 0^+} \| \Pi^\varepsilon_f(\rho) - \rho \|_X = 0$,
        \item\label{item:Gateaux}
        $f^\eps$ is Gâteaux differentiable on $X$ and for all $\rho \in X$ it holds $\{\rmr f^\eps(\rho) \} = \underline{\partial} f^\eps(\rho)$,
        \item\label{item:Frechet}
        if $X$ is a Hilbert space then $f^\eps$ is Fréchet differentiable on $X$,\item\label{item:StationarityProximalPoint}
        for any $\rho,\rho^\eps\in X$ the relation $\rho^\eps = \Pi_f^\eps(\rho)$ is equivalent to $\rmr f^\eps(\rho)=\frac{1}{\eps}J(\rho-\rho^\eps)\in\underline\partial f(\rho^\eps)$, 
        \item\label{item:limit-pot} for each $\rho \in \mathrm{dom}(\underline{\partial}f)$ there exists an element $v \in \underline{\partial}f(\rho)$ such that $\rmr f^\eps(\rho) = v^\eps \rightharpoonup v$ (weakly) as $\eps \to 0^+$, where $v = \argmin \qty{\norm{u}_{X^\ast} : u \in \underline{\partial}f(\rho)}$ is the element in $\underline{\partial} f(\rho)$ with minimal norm,
        \item\label{item:limit-pot-strong} if $X^\ast$ is uniformly convex then it holds for the convergence in \ref{item:limit-pot} that
        $v^\eps \to v$
        (strongly) as $\eps \to 0^+$.
    \end{enumerate} 
\end{theorem}
Points \ref{item:MY-prox-inequality} and especially \ref{item:limit-prox} show that the Moreau--Yosida regularisation can be viewed as a smooth approximation to a functional. Note that \ref{item:limit-prox} is just the statement of strong convergence $\Pi^\varepsilon_f(\rho) \to \rho$ as $\eps \to 0^+$.
Since every real Hilbert space is uniformly convex and uniformly smooth, we know from \ref{item:limit-pot-strong} that for real Hilbert spaces the $v^\varepsilon$ retrieved by the regularisation converges strongly to the element of $\underline{\partial}f(\rho)$ with minimal norm. This will be a crucial ingredient for the inversion scheme presented in \cref{sec:InversionScheme}.
Furthermore, the proximal mapping is \textit{firmly nonexpansive} on Hilbert spaces. The following result is a slight generalisation of \cite[Prop.~2.3~\&~4.2]{Bauschke_2017}.  

\begin{theorem}\label{thm:firmNE}
   Let $\mathcal{H}$ be a Hilbert space and suppose that $f: \mathcal{H} \to \overline{\mathbb{R}}$ is a proper, convex, and lower semicontinuous functional and $f^\eps: \mathcal{H} \to \overline{\mathbb{R}}$ is the corresponding Moreau--Yosida regularisation. Then for all $\eps>0$ the proximal mapping $\Pi_f^\eps : \mathcal{H} \to \mathcal{H}$ and $(\mathrm{Id} - \Pi^\varepsilon_f): \mathcal{H} \to \mathcal{H}$ are both firmly nonexpansive, i.e., for all $\rho,\sigma \in \mathcal{H}$
   \begin{align}\label{eq:firmNE}
       \norm{\Pi_f^\eps(\rho) - \Pi_f^\eps(\sigma)}_\mathcal{H}^2 &\leq \langle J(\rho - \sigma),  \Pi_f^\eps(\rho) - \Pi_f^\eps(\sigma)\rangle , \\
       \norm{[\mathrm{Id} - \Pi^\varepsilon_f](\rho) -[\mathrm{Id} - \Pi^\varepsilon_f](\sigma)}_\mathcal{H}^2 &\leq \langle J(\rho - \sigma),  [\mathrm{Id} - \Pi^\varepsilon_f](\rho) - [\mathrm{Id} - \Pi^\varepsilon_f](\sigma)\rangle  . \label{eq:firmNE2}
   \end{align}
\end{theorem}
\begin{proof}
    The first part of the statement can be found in \cite[Sec.~II.B]{Herbst_2025}\footnote{Note, however, that the proof in \cite{Herbst_2025} contains a typo, the subdifferential appearing in the proof should be $\underline \partial \mathcal{F}$ in the notation employed there.}.
    For the second part, we write $\rho^\eps = \Pi_f^\eps(\rho)$ and $\sigma^\eps = \Pi_f^\eps(\sigma)$, and use that $\tfrac{1}{\eps} J(\rho-\rho^\eps)\in\underline\partial f(\rho^\eps)$ from Theorem~\ref{thrm:MYRegProperties}\ref{item:StationarityProximalPoint}. It then follows with $A(\rho)=\rho-\rho^\eps$ 
    \[ \eps \left( \underline{\partial} f(\rho^\eps) - \underline{\partial} f(\sigma^\eps) \right) \ni J(A(\rho)) - J(A(\sigma) ) = J(A(\rho) -A(\sigma)). \]
    In the last step we have used the linearity of the duality map $J$. We then have
    by taking the dual pairing with $\rho^\eps - \sigma^\eps$   
    and using the  maximal monotonicity of the subdifferential
    \begin{equation*}
        0 \leq \langle  J(A(\rho) -A(\sigma)), \rho^\eps - \sigma^\eps \rangle
        = - \Vert A(\rho) -A(\sigma) \Vert_\mathcal{H}^2 + \langle  J(A(\rho) -A(\sigma)), \rho - \sigma \rangle .
    \end{equation*}
    This concludes the proof for the second part of the statement.
\end{proof}

Furthermore, by a direct application of the Cauchy--Schwarz inequality to \cref{eq:firmNE,eq:firmNE2} from Theorem~\ref{thm:firmNE}, it follows that both $\Pi^\varepsilon_f$ and $\mathrm{Id} - \Pi^\varepsilon_f$ are non-expansive (Lipschitz-continuous with constant 1), i.e., for all $\rho,\sigma \in \mathcal{H}$
\begin{equation*}
    \norm{\Pi_f^\eps(\rho) - \Pi_f^\eps(\sigma)}_\mathcal{H} \leq \norm{\rho - \sigma}_\mathcal{H},
    \quad \text{and} \quad 
    \norm{(\mathrm{Id} - \Pi_f^\eps)(\rho) - (\mathrm{Id} - \Pi_f^\eps)(\sigma)}_\mathcal{H} \leq \norm{\rho - \sigma}_\mathcal{H}.
\end{equation*}
This property will be of importance for the derivation of the error bounds for the Moreau--Yosida-based inversion scheme in the next section.

\subsection{Periodic Homogeneous Sobolev Spaces}\label{subsec:homogeneousSobolev}
Since we are interested in a density-potential inversion scheme for periodic materials, let us restrict ourselves to periodic systems on $\mathbb{R}^3$. We take a lattice $\mathcal{R}$ in $\mathbb{R}^3$ with reciprocal lattice $\mathcal{R}^*$ and unit cell $\Omega$~\cite{AshcroftMermin}. Recall, that a function $u\in L^2_\mathrm{per}(\Omega)$ that is periodic, i.e., translationally invariant under all lattice vectors $\RR\in\mathcal{R}$, can be expanded in a basis of plane waves $e_\GG(\rr)$ as 
\begin{equation}\label{eq:FourierRepresentation}
     u(\rr) = \sum_{\GG \in \mathcal{R}^\ast} \hat{u}_\GG e_\GG(\rr).
\end{equation}
Explicitly, the Fourier coefficients $\hat{u}_\GG$ and the Fourier basis functions $e_\GG(\rr)$ are given by
\begin{equation*}
    \hat{u}_\GG = \langle e_\GG,u \rangle = \int_\Omega  \overline{e_\GG}(\rr) u(\rr)  \dd{\mathbf{r}} , \quad e_\GG(\rr) =  \frac{1}{\sqrt{\abs{\Omega}}} \rme^{\rmi\GG\cdot \rr}.
\end{equation*}
The family $\qty{e_\GG}_{\GG \in \mathcal{R}^\ast}$ forms an orthonormal basis for the Hilbert space $L^2_\mathrm{per}(\Omega)$ of real, square-integrable functions.

We will turn our attention to the periodic variant of the usual Sobolev spaces $H^s(\Omega)$ that are also Hilbert spaces.
Sobolev spaces are chosen in order to penalise rapid oscillations, which will be beneficial for an inversion scheme targeting the exchange-correlation potential.
The \textit{periodic Sobolev spaces}~\cite{iorio2001fourier,Robinson2001-book} will, for any real $s \geq 0$, be defined in terms of their Fourier coefficients as
\begin{equation*}
	H^s_\mathrm{per}(\Omega) = \qty{ u : u(\rr) = \sum_{\GG \in \mathcal{R}^\ast} \hat{u}_\GG e_\GG(\rr), \overline{\hat{u}_\GG}=\hat{u}_{-\GG}, \sum_{\GG \in \mathcal{R}^\ast}\qty(1+\abs{\GG}^2)^s \abs{\hat{u}_\GG}^2 < \infty }.
\end{equation*}
Here the condition $\overline{\hat{u}_\GG}=\hat{u}_{-\GG}$ for the coefficients guarantees that the resulting function $u(\rr)$ is real. This definition can be generalised to $s<0$ and then includes distributions. The space $H^{-s}_\mathrm{per}(\Omega)$ is then exactly the topological dual to $H^s_\mathrm{per}(\Omega)$.
A particular extension of the Sobolev space with $s=1$ that will be useful for this work is the \textit{homogeneous Sobolev space}~\cite{mazya-book,galdi-book,ortner2012} of real, periodic functions.
\begin{equation}\label{eq:H1perhom}
    \Hperhom(\Omega) =\left. \qty{ u : u(\rr) = \sum_{\GG \in \mathcal{R}^\ast} \hat{u}_\GG e_\GG(\rr), \overline{\hat{u}_\GG}=\hat{u}_{-\GG}, \sum_{\GG \in \mathcal{R}^\ast}\abs{\GG}^{2}\abs{\hat{u}_\GG}^2 < \infty	} \middle/ \mathbb{R} \right. .
\end{equation}
The space is naturally equipped with the following norm that admits an equivalent position-space formulation using the weak derivative and the $L^2$-norm (we always take any representative $u\in[u]$ from the equivalence classes within the space and will later fix ourselves to the gauge $\hat u_{\GG=0}=0$),
\begin{equation*}
    \norm{u}_{\Hperhom} = \sqrt{\sum_{\GG \in \mathcal{R}^\ast}\abs{\GG}^{2} \abs{\hat{u}_\GG}^2} = \left( \int_\Omega \abs{\nabla u(\rr)}^2 \dd{\rr} \right)^{1/2}.
\end{equation*}
It is further a Hilbert space with inner product
\begin{equation*}
    \langle u,v \rangle_{\Hperhom} = \sum_{\GG \in \mathcal{R}^\ast}\abs{\GG}^{2} \overline{\hat{u}_\GG}\hat{v}_\GG = \int_\Omega \nabla u(\rr) \cdot \nabla v(\rr) \dd{\rr}.
\end{equation*}
We note that the zero Fourier coefficient $\GG=0$ does not contribute to the norm and thus the elements in $\Hperhom(\Omega)$ are only defined up to a constant shift. In order to still have a well-defined normed space where the condition
\begin{equation*}
    \|u\|_{\Hperhom}=0    
\end{equation*}
gives the unique zero element, the homogeneous Sobolev space had to be defined \emph{modulo} real numbers in \cref{eq:H1perhom}. This means the elements of $\Hperhom(\Omega)$ are equivalence classes of functions that all differ by a constant. This construction is ideally suited for the representation of potentials in our theory, that are also only defined up to a constant, which manifests the gauge freedom in energy measurements.

Next, we introduce the dual space
\begin{equation*}
    \Hperhomdual(\Omega)=(\Hperhom(\Omega))^*
\end{equation*}
that corresponds to the set of all continuous linear functionals on $\Hperhom(\Omega)$. We write the application of $f \in \Hperhomdual(\Omega)$ to a function $u \in \Hperhom(\Omega)$ in the form of the dual pairing $\langle u, f \rangle$ that gets evaluated as
\begin{equation}\label{eq:dual-pairing}
    f(u) = \langle u, f \rangle = \sum_{\GG \in \mathcal{R}^\ast} \overline{\hat{u}_\GG} \hat{f}_\GG.
\end{equation}
This yields the dual norm
\begin{equation}\label{eq:dual-norm}
    \norm{f}_{\Hperhomdual} = \sqrt{\sum_{\GG \neq 0} \frac{ |\hat{f}_\GG|^2 }{ \abs{\GG}^{2} }}
\end{equation}
and every $f\in\Hperhomdual(\Omega)$ is uniquely defined by a list of Fourier coefficients $\hat{f}_\GG$ that give a finite sum above. Note that the elements, as continuous linear functionals acting on a function space, have distributional character. This means that the Fourier series of a general $f\in\Hperhomdual(\Omega)$ does not necessarily converge to an integrable $f\in L^1(\Omega)$.
Like in the $\Hperhom(\Omega)$ norm, the Fourier coefficient corresponding to $\GG=0$ does not contribute in \cref{eq:dual-norm}, because we necessarily have $\hat f_{\GG=0}=0$ for every $f\in\Hperhomdual(\Omega)$. This must be the case because the $u\in\Hperhom(\Omega)$ are only defined up to a constant, which further implies that for any $c\in\mathbb{R}$ one has $\langle u+c,f \rangle = \langle u,f \rangle + c\langle 1,f \rangle = \langle u,f \rangle$ and thus $\langle 1,f \rangle=0$. By \cref{eq:dual-pairing} this means
\begin{equation*}
    \langle 1, f \rangle = \hat{f}_{\GG=0} = 0
\end{equation*}
for every $f\in\Hperhomdual(\Omega)$.
That elements in $\Hperhom(\Omega)$ are only defined up to a constant shift thus translates to the condition $\hat f_{\GG=0}=0$ on the dual side for elements in $\Hperhomdual(\Omega)$, which fits the definition of the norm in \cref{eq:dual-norm}.

The duality mapping $J:\Hperhomdual(\Omega) \to \Hperhom(\Omega)$ from Definition~\ref{def:DualityMapping} is given by 
\begin{equation}\label{def:duality-map}
    J(f)(\rr) 
    = \sum_{\GG \neq 0} \frac{\hat{f}_\GG}{\abs{\GG}^{2}} e_\GG(\rr)
    = \int_{\mathbb R^3} \frac{f(\mathbf{r}')}{4\pi\abs{\mathbf{r} - \mathbf{r}'}} \dd{\mathbf{r}'},
\end{equation}
where the integral over $\mathbb R^3$ holds for locally integrable $f$ and is meant to be taken over the periodic repetition of the unit cell $\Omega$.
The inverse duality mapping $J^{-1} : \Hperhom(\Omega) \to \Hperhomdual(\Omega)$ on the other hand is linked to the Poisson equation,
\begin{equation}\label{def:duality-map-inv}
    J^{-1}(v)(\rr) = \sum_{\GG\neq 0} \abs{\GG}^2 \hat{v}_\GG e_\GG(\rr) = -\nabla^2v(\rr),
\end{equation}
which means $J: f\mapsto v$ actually yields a solution to the Poisson equation under a compensating background charge.

The introduced spaces will be of particular interest in the following, since they provide the spaces for potentials and densities in the studied periodic setting. Here, $\Hperhom(\Omega)$ takes the role of potential space (potentials are only defined up to a constant), and by duality $\Hperhomdual(\Omega)$ takes the role of the density space.
In order not to have to work with equivalence classes, we will always select the unique representative that has $\hat{u}_{\GG=0}=0$ for any potential.
On the side of the densities an issue remains in that the condition $\hat f_{\GG=0}=0$ translates to $\int_\Omega f(\rr) \dd{\rr} = 0$ for integrable $f\in\Hperhomdual(\Omega)\cap L^1(\Omega)$. Yet, one usually requires that densities are normalised to the total number of particles $N$, e.g., in \cref{eq:def:rho}. This can simply be solved by adding the constant value $N/|\Omega|$ to every element of $\Hperhomdual(\Omega)$ which yields an affine space of densities that all integrate to $N$.
The periodic homogeneous Sobolev space and its dual will now become the spaces for formulating DFT in a periodic setting.

\begin{definition}\label{def:spaces}
Based on the Hilbert space
\begin{equation*}
    \Xdens = \Hperhomdual(\Omega)
    \;\text{with norm}\; \|f\|_\Xdens = \left\|f\right\|_{\Hperhomdual}= \sqrt{\sum_{\GG \neq 0} \frac{ |\hat{f}_\GG|^2 }{ \abs{\GG}^{2} }}
\end{equation*}
we define the affine space for densities
\begin{equation*}
    \Xdensaff = \left\{ f+\frac{N}{|\Omega|} : f\in\Xdens \right\} \;\text{with induced norm}\; \|\rho\|_{\Xdensaff} = \left\|\rho-\frac{N}{|\Omega|}\right\|_{\Xdens}= \sqrt{\sum_{\GG \neq 0} \frac{ |\hat{\rho}_\GG|^2 }{ \abs{\GG}^{2} }}.
\end{equation*}
For potentials, as the dual space, we fix the $\hat v_{\GG=0}=0$ gauge for its elements and define
\begin{equation*}
    \Xpot = \{ v : v\in[v]\in\Hperhom(\Omega), \hat v_{\GG=0}=0 \}
    \;\text{with norm}\; \|v\|_{\Xpot} = \|v\|_{\Hperhom} = \sqrt{\sum_{\GG \neq 0} \abs{\GG}^{2}|\hat{v}_\GG|^2}.
\end{equation*}
The definitions of the duality mapping and its inverse in \cref{def:duality-map,def:duality-map-inv} naturally extend to those spaces since the $\GG=0$ component is not considered.
Often we will encounter differences of densities $\rho,\rho'\in\Xdensaff$, in which case we naturally have $\rho-\rho'\in\Xdens$.
\end{definition}

We will conclude this section with some important results that relate the given spaces to the DFT setting and the Hamiltonian. First we discuss the mean-field approximation to the interaction energy of the particles in one unit cell that can be given purely in terms of the density norm and is called the Hartree energy. We then show that our new density space $\Xdensaff$ includes all previous densities from DFT and that all potentials from $\Xpot$ lead to a well-defined Hamiltonian that is self-adjoint and bounded below. This demonstrates that the spaces $\Xdensaff$ and $\Xpot$ seamlessly integrate into a periodic DFT setting.

As already noted in \cref{def:duality-map-inv}, the duality mapping yields the solution to the Poisson equation of electrostatics. The Hartree energy, defined as the mean-field electrostatic energy of a density, is then
\begin{equation}\label{eq:HartreeEnergy}
     E_\mathrm{H}(\rho) = \frac{1}{2}\langle J(\rho),\rho \rangle.
\end{equation}
Here, the prefactor is introduced to avoid double-counting. The Hartree energy is thus related to the norm on the chosen density space as given by the following proposition.  
\begin{proposition}\label{prop:EH}
    The Hartree energy is
    $E_\mathrm{H}(\rho) = \frac{1}{2}\|\rho\|_{\Xdensaff}^2$.
\end{proposition}
\begin{proof}
    For $\rho \in \Xdensaff$, recall that $\rho=f+N/|\Omega|$ where $f \in \Xdens$ has zero mean. From \cref{def:duality-map} it follows that   $\langle J(\rho), \rho \rangle = \langle J(f), f \rangle$, and \cref{eq:HartreeEnergy} implies that
    \begin{equation*}
        E_\mathrm{H}(\rho)  = \frac{1}{2}\langle J(f),f \rangle = \frac{1}{2}\|f\|_{\Xdens}^2 = \frac{1}{2}\|\rho\|_{\Xdensaff}^2
    \end{equation*}
\end{proof}
Note that by transitioning from $\rho$ to $f$, the divergent long-range contributions of the interaction energy are removed. By the definition of $J$ on $\Xdensaff$ as an extension from $\Xdens$ that removes the $\GG=0$ component, these long-range contributions are automatically taken care of.

\begin{lemma}\label{lemma:densSet-in-Xdensaff}
    The set $\densSet$ of $N$-representable densities is a subset of $\Xdensaff$.
\end{lemma}

\begin{proof}
    \citet{Lieb1983} demonstrated that $\densSet \subseteq L^1(\mathbb{R}^3)\cap L^3(\mathbb{R}^3)$ which also holds for a periodic setting with bounded domain $\Omega$ as $\densSet \subseteq L^3_\mathrm{per}(\Omega) \subseteq L^{6/5}_\mathrm{per}(\Omega)\subseteq L^1_\mathrm{per}(\Omega)$ and $L^1_\mathrm{per}(\Omega)\cap L^3_\mathrm{per}(\Omega) = L^3_\mathrm{per}(\Omega)$. Since $\Xpot\subseteq L^6_\mathrm{per}(\Omega)$ (\cite[Thm.~2.2(i)]{ortner2012}; the result on $\mathbb{R}^3$ also holds on the periodic $\Omega$ by simple restriction) which entails $L^{6/5}_\mathrm{per}(\Omega)\subseteq\Xdens$  and since all elements in $\densSet$ integrate to $N$, we have the desired inclusion in $\Xdensaff$.
\end{proof}

\begin{lemma}
    For any $v\in\Xpot$ the Hamiltonians $\hat{H} = \hat{T} + \hat{W} + \hat{V}$ and $\hat{H}_\mathrm{KS} = \hat{T} + \hat{V}$ with $\hat{V} = \sum_{j=1}^N v(\rr_j)$ both fulfil the conditions for the KLMN theorem and are thus self-adjoint and bounded below on the form domain $Q(\hat T)$.
\end{lemma}

\begin{proof}
    In the proof of Lemma~\ref{lemma:densSet-in-Xdensaff} it was noted that $\Xpot\subseteq L^6_\mathrm{per}(\Omega)$. But this potential space is known to be in the scope of the KLMN theorem (\cite[Thm.~X.17]{reed-simon-2} and even the Kato--Rellich theorem~\cite[Thm.~X.16]{reed-simon-2}) and so together with the standard assumption from \cref{subsec:DFT-KS}  that $\hat W$ also allows an application of the KLMN theorem.
\end{proof}

\section{The Moreau--Yosida inverse Kohn--Sham algorithm}\label{sec:InversionScheme}
We now present a general framework for density-potential inversion using the Moreau--Yosida regularisation that builds on \cite{Penz_2023} and that generalises the KS inversion algorithm presented in \cite{Herbst_2025}. Then, we exemplify the general setting using the periodic homogeneous Sobolev spaces and prove the lower semicontinuity of the mixed-state constrained-search functional for the KS system in this topology.

\subsection{The Inversion Algorithm}\label{subsec:InversionAlgorithm}
The idea of the inversion algorithm is to recover the potential from which a ground-state density originates. In KS inversion, the inversion algorithm amounts to recovering the exact exchange-correlation potential in the KS method, and this will be addressed in \cref{subsec:PeriodicIKS}.  However, we will first introduce our density-potential inversion framework in a general setting that allows application of the Moreau--Yosida regularisation~\cite{Penz_2025Perspective_arXiv}.
\begin{assumption} \label{assump:FX}
    The density space $X$ is a strictly convex and reflexive real Banach space that contains $\densSet$, and the dual space $X^\ast$ is strictly convex. The functional $\mathcal{F}$ is proper, convex, and lower semicontinuous (l.s.c.) with respect to the topology of $X$. 
\end{assumption}

\begin{definition}\label{def:representable}
    A density $\rho_\mathrm{gs}\in X$ is called \emph{$v$-representable with respect to $\mathcal{F}$} if $\underline{\partial} \mathcal{F}(\rho_\mathrm{gs})\neq \emptyset$. In addition, any element $v\in X^*$ in $-\underline{\partial} \mathcal{F}(\rho_\mathrm{gs})$ we call a \emph{representing potential}.
\end{definition}

Such a representing potential $v\in -\underline{\partial} \mathcal{F}(\rho_\mathrm{gs})$ then yields the corresponding ground-state density as the minimiser in the energy expression,
\begin{equation*}
    \rho_\mathrm{gs} = \argmin_{\rho\in X}(\mathcal{F}(\rho) + \langle v, \rho\rangle).
\end{equation*}
To determine a representing potential, we will use the proximal point of $\mathcal{F}$ and the duality mapping $J:X\to X^\ast$. The central problem of the inversion scheme is the minimisation over $\rho \in X$ of the energy functional
\begin{equation}\label{eq:MYEnergyFunctional}
    \mathcal{E}(\rho;\, \rho_\mathrm{gs}) = \mathcal{F}(\rho) + \frac{1}{2\eps} \norm{\rho-\rho_\mathrm{gs}}^2_{X}.
\end{equation}
By the fact that the functional $\rho \mapsto \norm{\rho}^2_{X}$ is strictly convex and continuous, it follows that $\rho \mapsto \mathcal{E}(\rho;\, \rho_\mathrm{gs})$ is a strictly convex and l.s.c.\ functional on $X$. Since the infimum of \cref{eq:MYEnergyFunctional} is precisely the Moreau envelope of $\mathcal{F}$ at $\rho_\mathrm{gs}$, the minimiser of $\mathcal{E}(\rho;\,\rho_\mathrm{gs})$ is the proximal point $\Pi^\varepsilon_{\mathcal{F}}(\rho_\mathrm{gs}) = \argmin_{\rho \in X} \mathcal{E}(\rho;\,\rho_\mathrm{gs})$. Recall from Proposition~\ref{prop:UniqueProx} that the proximal point exists uniquely for every $\rho_\mathrm{gs}\in X$ and we will simplify our notation by denoting the proximal point by $\rho^\eps = \Pi^\varepsilon_{\mathcal{F}}(\rho_\mathrm{gs})$ (c.f.\ \cref{subsec:MYreg}) when the argument and choice of functional is clear from context. We refer to $\rho^\eps$ as the \textit{proximal density}. Furthermore, it holds from Theorem~\ref{thrm:MYRegProperties}\ref{item:limit-prox} that  $\rho^\eps \to \rho_\mathrm{gs}$ as $\eps \to 0^+$ whenever $\rho_\mathrm{gs}\in \dom (\mathcal{F})$.
From an algorithmic point of view, the uniqueness of the proximal density is an advantageous feature, since it facilitates convergence of an iterative method for determining it.

Since $\underline{\partial} \tfrac{1}{2}\norm{\cdot}^2_X$ gives the duality mapping, the stationarity condition for the proximal density, $\underline{\partial}\mathcal{E}(\rho^\eps;\,\rho_\mathrm{gs}) \ni 0$, implies that (Theorem~\ref{thrm:MYRegProperties}\ref{item:StationarityProximalPoint})
\begin{equation*}
    \underline{\partial} \mathcal{F}\qty(\rho^\eps) + \frac{1}{\eps}J\qty(\rho^\eps - \rho_\mathrm{gs}) \ni 0. 
\end{equation*}
Under the assumption that $\underline{\partial} \mathcal{F}(\rho_\mathrm{gs})\neq\emptyset$ ($v$-representability with respect to $\mathcal{F}$),
one obtains a representing potential $v$ as the (weak) limit of $v^\eps = \frac{1}{\eps}J(\rho^\eps - \rho_\mathrm{gs})$  as $\eps \to 0^+$ by Theorem~\ref{thrm:MYRegProperties}\ref{item:StationarityProximalPoint}\&\ref{item:limit-pot}. In particular, this $v$ is the unique element in $-\underline{\partial}\mathcal{F}(\rho_\mathrm{gs})$ with minimal norm. If $X$ is a Hilbert space, then the convergence $v^\eps \to v$ is strong.
To summarise, a representing potential can be obtained as follows.
\begin{theorem}\label{thrm:DensPotInv}
    Suppose that $X$ and $\mathcal{F}$ fulfil Assumption~\ref{assump:FX}, and let $\rho_\mathrm{gs}\in X$ be $v$-representable with respect to $\mathcal{F}$, i.e., $\underline{\partial} \mathcal{F}(\rho_\mathrm{gs})\neq \emptyset$. 
    Let $v^\eps=\frac{1}{\eps}J\left(\Pi^\varepsilon_{\mathcal{F}}(\rho_\mathrm{gs}) - \rho_\mathrm{gs}\right)$,
    then a representing potential is given by
    \begin{equation*}
        v = \wlim_{\eps \to 0^+} v^\eps 
        = \wlim_{\eps \to 0^+} J\left( \frac{1}{\eps}\left(\Pi^\varepsilon_{\mathcal{F}}(\rho_\mathrm{gs}) - \rho_\mathrm{gs}\right)\right).
    \end{equation*}
    If $X= \mathcal{H}$ is a Hilbert space then
    \begin{equation}\label{eq:ProxRightDerivative}
        v = J\qty(\lim_{\eps \to 0^+}\frac{1}{\eps} \left(\Pi^\varepsilon_{\mathcal{F}}(\rho_\mathrm{gs}) -\rho_\mathrm{gs}\right)).
    \end{equation}
    For any $\tilde{\rho}_\mathrm{gs} \in \mathcal{H}$ and $\eps > 0$ let $ \tilde{v}^\eps = \tfrac{1}{\eps}J(\Pi^\varepsilon_{\mathcal{F}}(\tilde{\rho}_\mathrm{gs}) - \tilde{\rho}_\mathrm{gs}))$ then it holds
    \begin{equation}\label{eq:TotalErrorBound}
        \norm{v- \tilde{v}^\eps}_{\mathcal{H}^\ast} 
        \leq \norm{v - v^\eps}_{\mathcal{H}^\ast} + \frac{1}{\eps} \norm{\rho_\mathrm{gs} - \tilde{\rho}_\mathrm{gs} }_{\mathcal{H}}. 
    \end{equation}
\end{theorem}
\begin{proof}
    As already noted, the first formula follows from Theorem~\ref{thrm:MYRegProperties}\ref{item:StationarityProximalPoint}\&\ref{item:limit-pot}. 
    Here the prefactor $1/\eps$ can always be passed into the duality mapping because of its homogeneity.
    If further $X^*$ is uniformly convex, the convergence is strong (Theorem~\ref{thrm:MYRegProperties}\ref{item:limit-pot-strong}) and the duality mapping is absolutely continuous~\cite[Prop.~1.117]{Barbu-Precepanu}. The limit can then be passed to the argument of $J$ and \cref{eq:ProxRightDerivative} follows. Furthermore, by the linearity of the duality mapping on Hilbert spaces and the triangle inequality we have
    \begin{align*}
        \norm{v- \tilde{v}^\eps}_{\mathcal{H}^\ast}  \leq \norm{v - v^\eps}_{\mathcal{H}^\ast} + \frac{1}{\eps}\norm{\Pi^\varepsilon_{\mathcal{F}}(\rho_\mathrm{gs}) - \rho_\mathrm{gs} - \Pi^\varepsilon_{\mathcal{F}}(\tilde{\rho}_\mathrm{gs}) + \tilde{\rho}_\mathrm{gs}}_{\mathcal{H}}
    \end{align*}
    and by Theorem~\ref{thm:firmNE} it further holds 
    \begin{equation*}
        \norm{\rho^\eps - \rho_\mathrm{gs} - \tilde{\rho}^\eps + \tilde{\rho}_\mathrm{gs}}_{\mathcal{H}} 
        = \norm{\qty(\mathrm{Id} - \Pi^\varepsilon_{\mathcal{F}})( \rho_\mathrm{gs}) - \qty(\mathrm{Id} - \Pi^\varepsilon_{\mathcal{F}})(\tilde{\rho}_\mathrm{gs})}_{\mathcal{H}} 
        \leq \norm{\rho_\mathrm{gs} - \tilde{\rho}_\mathrm{gs} }_{\mathcal{H}}.
    \end{equation*}
    This shows the inequality in \cref{eq:TotalErrorBound}.
\end{proof}

The choice of the functional $\mathcal{F}$ in the presentation so far is only limited by Assumption~\ref{assump:FX}. Motivated by our applications in \cref{subsec:PeriodicIKS}, for a functional $\mathcal{F}_0$ that fulfils Assumption~\ref{assump:FX} and parameters
$\alpha\in \mathbb R$ and $v_0\in X^\ast$, we define
\begin{equation}\label{eq:variableFunc}
    \mathcal{F}_{\alpha,v_0}(\rho) = \mathcal{F}_0(\rho) +  \frac{\alpha}{2}\norm{\rho}_X^2   +    \langle v_0, \rho \rangle.
\end{equation}
The functional $\mathcal{F}_{\alpha,v_0}$ is then still proper, convex, and l.s.c., and thus satisfies Assumption~\ref{assump:FX}. In the Hilbert-space case we prove the following. 

\begin{proposition}\label{prop:ProxGuidingFunc}
    Let $X = \mathcal{H}$ be a real Hilbert space and suppose $\alpha\geq 0$, $v_0\in \mathcal{H}^\ast$, $\rho\in \mathcal{H}$ and set $\mu = \eps/(1 + \eps\alpha)>0$. 
    For $\mathcal{F}_{\alpha,v_0}$ defined as in \cref{eq:variableFunc} the proximal point relates to that of $\mathcal{F}_0$ as
    \begin{equation*}
       \Pi_{\mathcal{F}_{\alpha,v_0}}^\eps\qty(\rho) 
        = \Pi^\mu_{\mathcal{F}_0} \qty(\frac{\mu}{\eps}\rho - \mu J^{-1}(v_0)).
    \end{equation*}
\end{proposition}

\begin{proof}
    Inserting the definition \cref{eq:variableFunc} and using the expansion of the square
    \begin{equation*}
        \norm{\sigma - \rho}^2_\mathcal{H} = \norm{\sigma}_\mathcal{H}^2 - 2\langle J(\rho),\sigma\rangle + \norm{ \rho}_{\mathcal{H}}^2 ,
    \end{equation*}
    we obtain (ignoring terms constant in $\sigma$ that do not affect the $\argmin$)
    \begin{align*}
		\Pi_{\mathcal{F}_{\alpha,v_0}}^\eps (\rho) 
        &= \argmin_{\sigma \in \mathcal{H}} \qty{ \mathcal{F}_0(\sigma) + \frac{\alpha}{2}\norm{\sigma}^2_\mathcal{H}  + \langle v_0,  \sigma \rangle + 
        \frac{1}{2\eps}\norm{\sigma -\rho}^2_\mathcal{H} } \\
        &= \argmin_{\sigma \in \mathcal{H}} \qty{ \mathcal{F}_0(\sigma) + \qty(\frac{\alpha}{2}+\frac{1}{2\eps}) \norm{\sigma}^2_\mathcal{H} - \langle \eps^{-1} J(\rho)-v_0, \sigma \rangle  } \\
        &= \argmin_{\sigma \in \mathcal{H}} \qty{ \mathcal{F}_0(\sigma) + \frac{1}{2\mu}  \norm{\sigma - \mu\left(\eps^{-1}\rho-J^{-1}(v_0)\right)}^2_\mathcal{H}  }
        = \Pi_{\mathcal{F}_0}^\mu \qty(\frac{\mu}{\eps}\rho-\mu J^{-1}(v_0)).
	\end{align*}
\end{proof}

\subsection{The Kinetic-Energy Constrained-Search Functional}\label{subsec:kin-functional}
Throughout the remainder of this work, we assume the periodic setting established in \cref{subsec:homogeneousSobolev}. Recall that in this setting, the unit cell $\Omega$ is a bounded subset of $\mathbb{R}^3$ and (its closure) is always compact.
We choose the basic spaces for densities and potentials to be $\Xdensaff$ and $\Xpot$, as given in Definition~\ref{def:spaces}.  
In analogy to the functional $F_\mathrm{DM}(\rho)$, let us introduce the kinetic-energy mixed-state constrained-search functional $T_\mathrm{DM}: \Xdensaff \to \overline{\mathbb
R}$ as 
\begin{equation}\label{eq:def-T}
    T_\mathrm{DM}(\rho) = \left\{ \begin{array}{ll}
         \displaystyle \inf_{\Gamma\mapsto \rho} \Tr[\hat{T}\Gamma] & \text{if}\; \rho\in\densSet, \\
         +\infty & \text{otherwise}. 
    \end{array}\right.
\end{equation}
In the inversion algorithm of \cref{subsec:InversionAlgorithm}, if the functional $\mathcal{F}$ contains $T_\mathrm{DM}$ we refer to the density-potential inversion as a \emph{KS inversion scheme}. The kinetic-energy functional thus plays a central role for the inversion and we will next verify that it satisfies Assumption~\ref{assump:FX}.
Since the functional is defined as an infimum over a linear function and with a linear constraint, it is automatically convex on the convex set $\densSet \subset \Xdensaff$. The extension to all of $\Xdensaff$ by $+\infty$ conserves the convexity. The rest of this section will be devoted to showing that it is also l.s.c.\ with respect to the norm of $\Xdens$. The importance of lower semicontinuity for density functionals is highlighted in a recent perspective article~\cite{Penz_2025Perspective_arXiv}. A consequence for reflexive spaces $\Xdens$ discussed there is that $\overline{\partial}E(v)$ must be non-empty for every $v\in\Xpot$, which means that one finds a ground state for every potential. This naturally holds in a periodic setting.

The proof of lower semicontinuity of $T_\mathrm{DM}(\rho)$ is facilitated by showing first the lower semicontinuity of the von Weizsäcker kinetic-energy functional defined by 
\begin{equation}\label{eq:vWkef}
    G(\rho) = \left\{ 
        \begin{array}{ll}
            \displaystyle \int_\Omega \qty(\nabla\sqrt{\rho}(\mathbf{r}))^2 \dd{\mathbf{r}} & \;\text{if}\; \rho\in\densSet, \\
            +\infty & \;\text{otherwise.} 
        \end{array}\right.
\end{equation}
The proof of the lower semicontinuity of $G$ is based on considering a sequence  $\rho_n\to\rho$ in $\Xdens$, where, when necessary, we pass to subsequences. Note that convergence in the $\Xdens$-topology is meaningful for sequences in $\Xdensaff$ because the affine space inherits its topology from $\Xdens$. The argument proceeds in four steps: (i) By the definition of $G$, any sequence $\{\sqrt{\rho_n}\}$ with $\rho_n$ in the effective domain of $G$ will be bounded in $H^1_\mathrm{per}(\Omega)$. (ii) By compact embedding we obtain $L^q$ convergence, for $1\leq q < 6$. (iii) We can then connect back to the $\Xdens$ topology by the use of Hölder's inequality and a further embedding to identify the limit of the sequence as needed. (iv) Weak lower-semicontinuity of the functional $f\mapsto\|\nabla f\|_{L^2}^2$ then establishes the proof.

\begin{lemma}\label{lemma:G-lsc}
    The von Weizsäcker kinetic-energy functional
    $G:\Xdensaff\to \overline{\mathbb{R}}$ defined in \cref{eq:vWkef} 
    is convex and weakly lower semicontinuous in $\Xdens$.
\end{lemma}

\begin{proof}
    We know from Lieb that $G(\rho)=\int_\Omega (\nabla\sqrt{\rho})^2=\|\nabla\sqrt{\rho}\|_{L^2}^2$ is convex on (the convex set) $\densSet$~\cite{Lieb1983}. Since it is extended to $+\infty$ outside of $\densSet$, $G$ must be a convex functional. This part does not change in the periodic setting.

    For (strong) l.s.c.\ we need to show that for every $\rho_n\to\rho$ in $\Xdens$, we have $G(\rho)\leq \liminf_n G(\rho_n)$. If $\liminf_n G(\rho_n)=+\infty$ then nothing remains to be shown, so we limit ourselves to the case $g=\liminf_n G(\rho_n)< +\infty$ and restrict the sequence to all $\rho_n$ with $G(\rho_n) < g+\delta$ for some $\delta>0$ (e.g., a subsequence that realises the limit inferior in its limit). This means $\rho_n\in\densSet$ and from $G(\rho_n) = \|\nabla\sqrt{\rho_n}\|_{L^2}^2 = \|\sqrt{\rho_n}\|_{H^1}^2-N$ (since $\|\sqrt{\rho_n}\|_{L^2}^2=\|\rho_n\|_{L^1}=N$) we then have $\{\sqrt{\rho_n}\}$ bounded in $H^1_\mathrm{per}(\Omega)$.
    Since $\Omega$ is bounded we can use the Rellich--Kondrachov theorem~\cite[Thm.~6.3]{adams-book} to show that $H^1_\mathrm{per}(\Omega)$ \emph{compactly} embeds into $L^q_\mathrm{per}(\Omega)$ with $1\leq q < 6$ and it still embeds continuously for $1\leq q\leq 6$. This means $\{\sqrt{\rho_n}\}$ is also bounded in $L^q_\mathrm{per}(\Omega)$, $1\leq q\leq 6$.
    Returning to the fact that the sequence $\{\sqrt{\rho_n}\}$ is bounded in $H^1_\mathrm{per}(\Omega)$, the compact embedding implies that there is a strongly convergent subsequence (again denoted by $\rho_n$) $\sqrt{\rho_n}\to f$ in $L^q$, $1 \leq q<6$. With the abbreviation $u_n=\sqrt{\rho_n}+f$ and using the Cauchy--Schwarz inequality we write
    \begin{equation*}
        \|\rho_n- f^2\|_{L^1} = \int_\Omega|\sqrt{\rho_n}-f|\cdot|u_n| \leq \|\sqrt{\rho_n}-f\|_{L^2}\cdot\|u_n\|_{L^2}.
    \end{equation*}
    Since in particular $\sqrt{\rho_n}\to f$ in $L^2$ and $u_n$ is bounded in $L^2$, this implies $\rho_n\to f^2$ in $L^1$, so $f^2$ also integrates to the particle number $N$.
    By the Banach--Alaoglu theorem for a Hilbert-space setting~\cite[Cor.~11.9]{clason2020-book} we also have that $\sqrt{\rho_n}\weakto f$ (weakly) in $H^1_\mathrm{per}(\Omega)$.
    With the continuous embedding $H^1_\mathrm{per}(\Omega)\hookrightarrow L^6_\mathrm{per}(\Omega)$ it follows $f^2\in L^3_\mathrm{per}(\Omega)$. We thus know that $f^2\in\densSet$ and with this $G(f^2)<\infty$.
    
    Now we use the definition of the dual norm in $\Xdens$ that takes the supremum over all dual pairings with elements from the predual, that is $\Xpot$ by reflexivity. We can write this dual pairing as an integral and in a next step apply the Cauchy--Schwarz inequality to obtain
    \begin{equation*}
        \|\rho_n- f^2\|_{\Xdens} = \sup\left\{ \int_\Omega (\sqrt{\rho_n}-f)u_n v : v\in \Xpot, \|v\|_{\Xpot}\leq 1\right\}\leq \|\sqrt{\rho_n}- f\|_{L^2} \cdot \sup\left\{\|u_n v\|_{L^2} : v\in \Xpot, \|v\|_{\Xpot}\leq 1\right\} .
    \end{equation*}
    We then use the Hölder inequality, $\|u_n v\|_{L^2}\leq \|u_n\|_{L^3} \cdot \|v\|_{L^6}$, and the continuous embedding $\Xpot \hookrightarrow L^6_\mathrm{per}(\Omega)$ that gives $\|v\|_{L^6}\leq C\|v\|_{\Xpot}$ to conclude
    \begin{equation*}
        \begin{aligned}
            \|\rho_n- f^2\|_{\Xdens} \leq C \|\sqrt{\rho_n}- f\|_{L^2} \cdot \|u_n\|_{L^3} \longrightarrow 0 .
        \end{aligned}
    \end{equation*}
    The limit to zero follows from the fact that as shown before $u_n$ is bounded in $L^q_\mathrm{per}(\Omega)$, $1 \leq q\leq 6$, and $\sqrt{\rho_n}\to f$ in $L^q$, with $q = 2 \in [1, 6)$. We have thus established that $\rho_n\to  f^2$ in $\Xdens$, while from the starting assumption it holds $\rho_n\to \rho$ in $\Xdens$, so it follows $ f^2= \rho$ and equivalently 
    $f= \sqrt{\rho}\in H^1_\mathrm{per}(\Omega)$. With $\sqrt{\rho_n}\weakto \sqrt{\rho}$ in $H^1_\mathrm{per}(\Omega)$ (from Banach--Alaoglu), we can now use that $f\mapsto\|\nabla f\|_{L^2}^2$ is (strongly) continuous in $H^1_\mathrm{per}(\Omega)$, which also makes it weakly l.s.c.\ as a convex function~\cite[Cor.~3.9]{brezis-book}. We thus obtain $G(\rho)\leq \liminf_n G(\rho_n)$ and strong l.s.c. Because, again, strong l.s.c.\ and convexity implies weak l.s.c., this concludes the proof.
\end{proof}

Before proving lower semicontinuity of $T_\mathrm{DM}(\rho)$ in the $\Xdens$-topology, one of the main results in this work, we make some remarks. An important ingredient of our proof is the continuous embedding
\begin{equation}\label{eq:H1L1inclusion}
       H^1_\mathrm{per}(\Omega) \subset L^6_\mathrm{per}(\Omega) \subset L^1_\mathrm{per}(\Omega), 
\end{equation}
that we already used in the proof of Lemma~\ref{lemma:G-lsc} and 
which can be used to reduce the setting to the $L^1$-topology considered by \citet[Th.~4.4]{Lieb1983}. 
Additionally, for all $\rho\in\densSet$ we have the estimate
    \begin{equation}\label{eq:T-estimates}
        \frac{1}{2}G(\rho) \leq T_\mathrm{DM}(\rho) \leq 8\pi^2N^2 G(\rho)
    \end{equation}
that follows from Theorems 1.1 and 1.2 in \citet{Lieb1983}, and which does not change in the periodic setting.
This estimate is useful to devise the proof in terms of the von Weizs\"acker kinetic-energy functional instead of $T_\mathrm{DM}(\rho)$ itself. With an additional estimate for the interaction part, it is also possible to show l.s.c.\ of the Lieb functional $F_\mathrm{DM}(\rho)$, but this is not further discussed here since $F_\mathrm{DM}(\rho)$ is not used for the KS inversion procedure.

\begin{theorem}\label{th:FDM-lsc}
    $T_\mathrm{DM}(\rho)$ is convex and weakly (and hence also strongly) lower semicontinuous (in $\Xdens$).
\end{theorem}

\begin{proof}
    The convexity was already demonstrated after \cref{eq:def-T}.
    For (strong) l.s.c., consider a sequence $\rho_n\to\rho$ in $\Xdens$. If $\rho\notin\densSet$ then by Lemma~\ref{lemma:G-lsc} we know that $\liminf_n G(\rho_n)=+\infty$, and the lower estimate in \cref{eq:T-estimates} implies that $\liminf_n T_\mathrm{DM}(\rho_n)=+\infty$ and we are done. Next, consider $\rho\in\densSet$, in which case we can also limit the sequence to $\rho_n\in\densSet$, because else $T_\mathrm{DM}(\rho) \leq T_\mathrm{DM}(\rho_n)=+\infty$. Setting $g=\liminf_n T_\mathrm{DM}(\rho_n)$, we can further restrict the sequence to $\rho_n$ with $T_\mathrm{DM}(\rho_n)<g+\delta$ for some $\delta > 0$. Now, we have $g+\delta > T_\mathrm{DM}(\rho_n)\geq G(\rho_n)=\|\sqrt{\rho_n}\|_{H^1_\mathrm{per}}^2-N$, so $\{\sqrt{\rho_n}\}$ is bounded in $H^1_\mathrm{per}(\Omega)$. By the same arguments as in the proof of Lemma~\ref{lemma:G-lsc} and the inclusion in \cref{eq:H1L1inclusion}, we obtain $\rho_n\to\rho$ in $L^1$.
    We are now in the setting of \citet[Thm.~4.4]{Lieb1983} (actually, weak convergence would be enough) that ensures that there is a density matrix $\Gamma$, $\Gamma\mapsto\rho$, such that $\Tr[ H_0\Gamma ] \leq \liminf_n T_\mathrm{DM}(\rho_n)$. By definition of $T_\mathrm{DM}(\rho)$, \cref{eq:def-T}, as a constrained-search functional, it then follows $T_\mathrm{DM}(\rho)\leq \Tr[ H_0\Gamma ]$ and consequently $T_\mathrm{DM}(\rho)\leq \liminf_n T_\mathrm{DM}(\rho_n)$. This demonstrates strong l.s.c., which along with convexity implies weak l.s.c.
\end{proof}

\subsection{Inverse Kohn--Sham on Periodic Homogeneous Sobolev Spaces} \label{subsec:PeriodicIKS}

Using Theorem~\ref{th:FDM-lsc}, we now wish to apply our general density-potential inversion scheme in \cref{subsec:InversionAlgorithm} to the periodic homogeneous Sobolev spaces. This corresponds to a Kohn--Sham setting in which the exchange-correlation potential is determined because the l.s.c.\ of $T_\mathrm{DM}$ establishes the link to the underlying non-interacting Schrödinger equation~\cite{Penz_2025Perspective_arXiv}.
In other words, the particular choices $X = \Xdens$ and $\mathcal F_0 = T_\mathrm{DM}$ satisfy Assumption~\ref{assump:FX}. Note that the density normalised to the number of particles is then obtained by the shift to $\Xdensaff$, the actual domain of the functional. If $\rho_\mathrm{gs} \in \Xdensaff$ is an interacting ground-state density with $\underline{\partial} T_\mathrm{DM}(\rho_\mathrm{gs}) \neq \emptyset$, then there exist at least one representing potential (Definition~\ref{def:representable}) that can be associated with a Schrödinger equation and we say that $\rho_\mathrm{gs}$ is \emph{non-interacting $v$-representable}. Moreover, any element of $-\underline{\partial} T_\mathrm{DM}(\rho_\mathrm{gs}) $ is said to be a KS potential and is denoted by  $v_\mathrm{KS}$.

Since the goal of the inversion algorithm is to target the xc potential, let us choose a \textit{guiding functional} that contains as much prior knowledge of the system as possible in order to aid the inversion algorithm. Recall that $E_\mathrm{H}(\rho) = \frac{1}{2} \norm{\rho}_{\Xdensaff}^2$ from Proposition~\ref{prop:EH}, then it immediately follows from \cref{eq:variableFunc} (with $\alpha=1$ and $v_0 = v_\mathrm{ext}\in \Xpot$) that the functional
\begin{equation} \label{eq:GuidingFunctional}
    \mathcal{F}(\rho) = T_\mathrm{DM}(\rho) + E_{\mathrm{H}}(\rho) + \langle v_\mathrm{ext},  \rho \rangle
\end{equation}
satisfies Assumption~\ref{assump:FX} and fits to the inversion scheme presented in \cref{subsec:InversionAlgorithm}. Therefore, we make $\mathcal{F} : \Xdens \to \overline{\mathbb{R}}$ the initial choice of functional to guide the inversion as it contains all contributions to \cref{eq:DensityFunctionalDecomp} except of the exchange-correlation part. However, other choices are possible and will be discussed in relation to Corollary~\ref{cor:ProxGuidingFunc}.

The central problem of the inversion scheme is the determination of the proximal density $\rho^\eps = \Pi^\varepsilon_\mathcal{F}(\rho_\mathrm{gs})\in \Xdensaff$ and the stationarity condition for the proximal density is  
\begin{equation*}
    \underline{\partial} \mathcal{F}(\rho^\eps) + \frac{1}{\eps}J(\rho^\eps - \rho_\mathrm{gs}) \ni 0 
    \quad \Longrightarrow \quad
    -\underline{\partial} T_\mathrm{DM}(\rho^\eps)\ni  v_\mathrm{H}(\rho^\eps) + v_\mathrm{ext}+ \frac{1}{\eps}J(\rho^\eps - \rho_\mathrm{gs}) =: v_\mathrm{KS}^\eps .
\end{equation*}
The limit procedure from Theorem~\ref{thrm:MYRegProperties}\ref{item:limit-pot-strong} then allows to pass to a $v_\mathrm{KS}=v+v_\mathrm{H}+v_\mathrm{xc}$.
The following result is a direct consequence of Theorems~\ref{thrm:DensPotInv} and~\ref{thrm:MYRegProperties} and serves to summarise the density-potential inversion scheme in the present case.
\begin{corollary}\label{cor:DensityPotentialIvesion}
    Suppose that $\mathcal{F}: \Xdens \to \overline{\mathbb{R}}$ is given by \cref{eq:GuidingFunctional} and let $\rho_\mathrm{gs}\in \Xdensaff$ be non-interacting $v$-representable. Then the proximal density is a unique point $\Pi^\varepsilon_\mathcal{F}(\rho_\mathrm{gs})$ that converges (strongly) to $\rho_\mathrm{gs}$ in $\Xdensaff$. The exchange-correlation potential
    \begin{equation}\label{eq:XdensVxc}
        v_\mathrm{xc} = \lim_{\eps \to 0^+} v_\mathrm{xc}^\eps 
        = \lim_{\eps \to 0^+} \frac{1}{\eps}J(\Pi^\varepsilon_\mathcal{F}(\rho_\mathrm{gs}) - \rho_\mathrm{gs}))
        = J\qty(\lim_{\eps \to 0^+} \frac{1}{\eps}\left(\Pi^\varepsilon_\mathcal{F}(\rho_\mathrm{gs}) -\rho_\mathrm{gs}\right))
    \end{equation}
    is the element in $-\underline{\partial}\mathcal{F}(\rho_\mathrm{gs})$ with minimal norm in $\Xpot$. Additionally, for all $\tilde{\rho}_\mathrm{gs} \in \Xdensaff$ and $ \tilde{v}_\mathrm{xc}^\eps =  \tfrac{1}{\eps}J(\Pi^\varepsilon_\mathcal{F}(\tilde{\rho}_\mathrm{gs}) - \tilde{\rho}_\mathrm{gs}))$we have for any $\eps > 0$ that
    \begin{equation}\label{eq:XdensTotalErrorBound}
        \norm{v_\mathrm{xc}- \tilde{v}_\mathrm{xc}^\eps}_{\Xpot} 
        \leq \norm{v_\mathrm{xc} - v_\mathrm{xc}^\eps}_{\Xpot} + \frac{1}{\eps} \norm{\rho_\mathrm{gs} - \tilde{\rho}_\mathrm{gs} }_{\Xdens}. 
    \end{equation}
\end{corollary}

\begin{remark}\hfill
    \begin{itemize}
        \item[(i)] From \cref{eq:XdensVxc} we identify the one-sided limit of $(\rho^\eps -\rho_\mathrm{gs})/\eps$ as the right derivative with respect to $\eps$ (denoted $\partial^+_\eps$) of the proximal density at $\eps=0$ and we can thus write $v_\mathrm{xc} = J(\left. \partial^+_\eps \rho^\eps\right|_{\eps=0})$. 
        \item[(ii)] Let $\tilde{\rho}_\mathrm{gs} = \rho_\mathrm{gs} + \Delta \rho$ be an inexact or perturbed ground-state density used as input in the inversion scheme. Then \cref{eq:XdensTotalErrorBound} provides a bound on the error in the potential obtained from the inversion at a non-zero $\eps$. 
        The first term in the right-hand side arises from terminating the procedure at some $\eps>0$ and goes to zero as $\eps\to 0^+$ by the strong convergence $v^\varepsilon_\mathrm{xc} \to v_\mathrm{xc}$ (Theorem~\ref{thrm:MYRegProperties}\ref{item:limit-pot-strong}), whereas the second term, $\tfrac{1}{\eps}\norm{\Delta \rho}_\Xdens$, blows up in the regime $\eps \ll \norm{\Delta \rho}_\Xdens$.
        The combination of two terms with opposite effects suggests that there might exist an optimal non-zero value for $\eps$ in practical calculations.
        Without regularisation, in the standard DFT formulation,  the variation of the potential cannot be controlled by the variation of the density. More precisely, Garrigue~\cite{Gar21} has demonstrated within Lieb's DFT framework ($L^p$-spaces) that the inverse problem is ill-posed. However, in a one-dimensional setting, and with a properly defined density space, the mapping from densities to potentials can in fact be shown to be Lipschitz continuous, i.e., density-potential inversion is stable in such a setting~\cite{Corso_2025-HK}. 
        \item[(iii)] \label{rmk:ZMP} We further remark that \cref{eq:XdensVxc} with \cref{def:duality-map} can be written as
        \begin{equation}\label{eq:ZMPconnection}
            v_\mathrm{xc}(\rr) = \lim_{\lambda \to \infty} \lambda \int_{\mathbb{R}^3} \frac{ \rho^{1/\lambda}(\rr') - \rho(\rr')}{4\pi|\rr - \rr'|} \dd{\mathbf{r}'}, \qquad \rho^{1/\lambda} = \Pi^{1/\lambda}_\mathcal{F}(\rho_\mathrm{gs}),
        \end{equation}
        where $\lambda = 1/\eps$ relabels the regularisation parameter $\eps$ in terms of the 
        (ZMP) parameter $\lambda$. For the density space $\Xdensaff$, this establishes the mathematical validity of the numerical approach due to \citet{ZMP1994} in the periodic setting. That the ZMP method can be understood in terms of Moreau--Yosida regularisation was first discussed in \cite{Penz_2023} 
        and an expression similar to \cref{eq:ZMPconnection} was obtained.
    \end{itemize}
\end{remark}

Using the construction of \cref{eq:variableFunc}, the guiding functional (\cref{eq:GuidingFunctional}) can be generalised to include variable amounts of the Hartree energy and external potential. In particular, suppose that $\mathcal{F}_0(\rho) =T_\mathrm{DM}(\rho)$ and that for $\alpha,\beta,\gamma\in\mathbb R$ and $\alpha\geq 0$ we choose $v_0 = \beta v_\mathrm{H}(\rho_\mathrm{gs}) + \gamma v_\mathrm{ext}$ in \cref{eq:variableFunc}, where we recall that $ v_\mathrm{H}(\rho_\mathrm{gs}) \in \Xpot$ is the Hartree potential and $v_\mathrm{ext}\in \Xpot$ is an external potential. We then define the $(\alpha,\beta,\gamma)$-parametrised guiding functional
\begin{equation}\label{eq:variableGuideFunc}
    \mathcal{F}_{\alpha,\beta,\gamma}(\rho) = T_\mathrm{DM}(\rho) + \alpha E_\mathrm{H}(\rho)  + \beta \langle v_\mathrm{H}(\rho_\mathrm{gs}),  \rho \rangle+  \gamma \langle v_\mathrm{ext}, \rho \rangle,
\end{equation}
which always satisfies Assumption~\ref{assump:FX}.
The proximal points of this guiding functional for two different sets of parameters can be related using Proposition~\ref{prop:ProxGuidingFunc} and the fact that $E_\mathrm{H}(\rho) = \frac{1}{2} \|\rho\|^2_\Xdens$ and $v_\mathrm{H}(\rho) = \rmr E_\mathrm{H}(\rho) = J(\rho)$.  We summarise in the following corollary.
\begin{corollary}\label{cor:ProxGuidingFunc}
    Let $\mathcal{F}_{\alpha,\beta,\gamma}$ be the guiding functional of \cref{eq:variableGuideFunc} with $\rho_\mathrm{gs}\in \Xdensaff$ and $v_\mathrm{ext}\in \Xpot$. Then for two parameter sets $(\alpha,\beta,\gamma)$ and $(\alpha',\beta',\gamma')$ with $\alpha>\alpha'\geq 0$  the proximal points are related by
    \begin{equation*}
        \Pi^\varepsilon_{\mathcal{F}_{\alpha,\beta, \gamma}}\qty(\rho) 
        = \Pi^\mu_{\mathcal{F}_{\alpha',\beta',\gamma'}} \qty(\frac{\mu}{\eps}\rho - \mu(\beta -\beta')\rho_\mathrm{gs}  -\mu(\gamma - \gamma') J^{-1}(v_\mathrm{ext})) \quad \text{for} \quad
        \mu = \frac{\eps}{1 + \eps(\alpha - \alpha')}>0.
    \end{equation*}
\end{corollary}

From the observation that the stationarity conditions for the proximal points corresponds to two different sets of parameters in \cref{eq:variableGuideFunc} and that the proximal densities are equal under Corollary~\ref{cor:ProxGuidingFunc}, we give the following remark. 
\begin{remark}
    The total KS potential determined by the inversion scheme does not depend on the choice of parameters in the guiding functional. In particular, for any choice of parameters
    $\alpha,\beta,\gamma\in\mathbb R$ with $\alpha\geq0$ where $\rho^\eps = \Pi^\varepsilon_{\mathcal{F}_{\alpha,\beta,\gamma}} (\rho_\mathrm{gs})$, the ($\eps$-dependent) potential determined by the inversion satisfies
    \begin{equation*}
        \underline{\partial} T(\rho^\eps) + v^\varepsilon_{\alpha,\beta,\gamma} \ni 0 
        \quad \text{where} \quad
        v^\varepsilon_{\alpha,\beta,\gamma} = \alpha v_\mathrm{H}[\rho^\eps] + \beta v_\mathrm{H}[\rho_\mathrm{gs}] + \gamma v_\mathrm{ext} + \frac{1}{\eps} J( \rho^\eps -\rho_\mathrm{gs}).
    \end{equation*}
    By Theorem~\ref{thrm:MYRegProperties}\ref{item:limit-pot-strong},  $v_{\alpha,\beta,\gamma}^\eps \to v_\mathrm{KS} $ (strongly) in $\Xpot$ (a Hilbert space), where $-v_\mathrm{KS}\in \underline{\partial} T(\rho_\mathrm{gs})$ is the element in the subdifferential with minimal norm.
\end{remark}
\begin{remark}
    The choice of guiding functional, i.e., the choice of parameters, will determine which potential the inversion algorithm is targeting. In particular, as $\eps \to 0^+$, the term
    $\tfrac{1}{\eps} J( \rho^\eps -\rho_\mathrm{gs})$ gives:
    \begin{enumerate}[(i)]
        \item the full KS potential if $(\alpha,\beta,\gamma)=(0,0,0)$, 
        \item the Hartree-exchange-correlation potential if $(\alpha,\beta,\gamma)=(0,0,1)$, and
        \item the exchange-correlation potential if $(\alpha,\beta,\gamma)=(1,0,1)$ or $(0,1,1)$.
    \end{enumerate}
\end{remark}
Since determining the xc potential is the major challenge for KS-DFT, the guiding functional usually includes the contributions from $v_\mathrm{ext}$ and $v_\mathrm{H}$. 

\section{Numerical Examples}\label{sec:Numerics}
Let us now illustrate the density-potential inversion scheme presented in this work, in particular the KS inversion scheme in the periodic setting presented in \cref{subsec:PeriodicIKS}, with some numerical examples. We focus on two examples: the periodic Gross--Pitaevskii equation in one dimension and a selection of three-dimensional bulk materials. The implementation and numerical results are available on GitHub\footnote{\href{https://github.com/vebjorhb/MY-periodic-inversion}{github.com/vebjorhb/MY-periodic-inversion}}.
First, let us briefly introduce the numerical implementation of the inversion algorithm.             

\subsection{Implementation of the Inversion Algorithm}\label{subsec:InversionApproaches}
A central problem in the presented inversion scheme is the determination of the
proximal density. Since, by Proposition~\ref{prop:UniqueProx}, the proximal density is
unique, it may be determined by any optimisation method of choice.
For applications in quantum chemistry the direct minimisation and self-consistent field approaches
are widely available in numerical codes and thus natural candidates.
We will further introduce a simple proximal point algorithm, the method typically used
for non-smooth and constrained minimisation problems~\cite{Parikh2014}.  

\subsubsection{Direct Minimisation Approach}\label{subsec:DirectMinimisation}
The central computational problem of the inversion algorithm outlined in \cref{sec:InversionScheme}
is to compute the proximal density for a given $\eps$. This can be achieved
by directly minimising the functional $\mathcal{E}(\rho;\, \rho_\mathrm{gs})$, as introduced in \cref{eq:MYEnergyFunctional}, with respect to $\rho$.
Specialising to the homogeneous Sobolev setting and 
the parametrised guiding functional of \cref{eq:variableGuideFunc},
this leads to the minimisation problem
\begin{equation}
    \label{eq:dmouterprob}
    \rho^\eps = \argmin_{\rho\in \Xdensaff} \qty{\mathcal{F}_{\alpha,\beta,\gamma}(\rho)
    + \frac{1}{2\eps} \norm{\rho-\rho_\mathrm{gs}}^2_{\Xdens}}.
\end{equation}
Instead of treating the full mixed-state constrained-search functional $T_\mathrm{DM}$
in \cref{eq:variableGuideFunc}, we approximate it by the KS kinetic-energy functional,
which is obtained by restricting $T_\mathrm{DM}$ to the pure-state constrained search over Slater determinants.
That is, if $\{\varphi^j\} \subset \mathcal{H}_1$ is an orthonormal family of
$N$ single-particle states, then the KS kinetic-energy functional
$T_\mathrm{S} : \Xdensaff \to \overline{\mathbb{R}}$ is 
\begin{equation*}
    T_\mathrm{S}(\rho) = \left\{ \begin{array}{ll}
         \displaystyle \inf \qty{ \sum_{j=1}^N \langle \hat{t}\varphi^j,  \varphi^j\rangle : \sum_{j=1}^N \abs{\varphi^j(\mathbf{r})}^2 = \rho, \, \varphi^j \in L^2_\text{per}(\Omega),\, \langle\varphi^j, \varphi^k\rangle = \delta_{j,k},\,\forall j,k= 1,\dots, N} & \text{if}\; \rho\in\densSet, \\
         +\infty & \text{otherwise} . 
    \end{array}\right.
\end{equation*}
That every $\rho\in\densSet$ can be represented by a such a Slater-determinant state was demonstrated in \citet[Th.~1.2]{Lieb1983}. 
By the fact that $T_\mathrm{S}$ is a constrained search restricted to Slater determinants, it follows that $T_\mathrm{DM}(\rho)\leq T_\mathrm{LL}(\rho)  \leq T_\mathrm{S}(\rho)$ for all $\rho \in \Xdensaff$. Note however, that $T_\mathrm{S}$ is in general not a convex functional.
Here, we consider $\{\varphi^j\}_{j=1}^N$ as a set of orbitals with unit occupations and any set of optimisers $\{\varphi^j_\star\} \subset \mathcal{W}_1$
of $T_\mathrm{S}(\rho)$ has finite kinetic energy by construction.  

Observe that by employing this approximation of the kinetic energy functional
the problem in \cref{eq:dmouterprob} is in fact a nested minimisation:
the evaluation of $T_\mathrm{S}$ requires minimising over orthonormal spin orbitals $\Phi = (\varphi^1, \ldots, \varphi^{N})$
that generate a density $\rho$, while the outer problem in \cref{eq:dmouterprob}
is the minimisation over $\rho\in\Xdensaff$.
The two optimisation problems can thus be merged, namely
by expressing the density in terms of orbitals,
\begin{equation*}
    \rho_\Phi(\mathbf{r}) = \sum_{j=1}^{N} \abs{\varphi^j(\mathbf{r})}^2,
\end{equation*}    
This leads to a joint minimisation problem over all orbitals,
\begin{equation}
    \label{eq:dmprob}
        \min\left\{
            \sum_{j=1}^N \langle \hat{t}\varphi^j,  \varphi^j\rangle
            + \mathcal{G}_{\alpha,\beta,\gamma}^\eps(\rho_\Phi; \rho_\text{gs}) \ :\ 
            \Phi = (\varphi^1, \ldots, \varphi^N) \in \left(L^2_\text{per}(\Omega)\right)^N,
            \langle\varphi^j, \varphi^k\rangle = \delta_{j,k},\, \forall j,k= 1,\dots, N
            \right\},
\end{equation}
where 
\[
    \mathcal{G}_{\alpha,\beta,\gamma}^\eps(\rho; \rho_\text{gs}) = 
    \alpha E_\mathrm{H}(\rho)  + \beta \langle v_\mathrm{H}(\rho_\mathrm{gs}),  \rho \rangle+  \gamma \langle v_\mathrm{ext}, \rho \rangle
    + \frac{1}{2\eps} \norm{\rho-\rho_\mathrm{gs}}^2_{\Xdens}.
\]
We will refer to \cref{eq:dmprob} to as the \textit{direct minimisation} problem.
For a given $\eps > 0$ and $\rho_\mathrm{gs} \in \Xdensaff$
we can obtain the proximal density as $\rho^\eps = \rho_{\Phi_\star}$,
where $\Phi_\star$ is a minimiser of \cref{eq:dmprob}.
Although the minimising orbitals are only unique up to unitary transformations on the subspace spanned by $(\varphi^1_\star,\ldots,\varphi^N_\star)$,
the problem is well-defined because all minimisers $\Phi_\star$ generate the same proximal density.
This general prescription for determining the proximal point has previously been employed in \cite{Herbst_2025}. 

\subsubsection{The Self-Consistent Field Approach} \label{subsec:SCFApproach}
The \textit{self-consistent field} (SCF) approach is an alternative iterative procedure
to the direct minimisation setting for computing the proximal density.
The same approximation to the ($\alpha,\beta,\gamma$)-parametrised guiding functional \cref{eq:variableGuideFunc} is used, namely the KS kinetic-energy functional $T_\mathrm{S}$ is employed, rather than $T_\mathrm{LL}$ or $T_\mathrm{DM}$. Again, we
consider an orbital parametrisation of the electronic density with $1 \leq j \leq N$
the summation index over the orbitals.

The proximal point may be computed by solving the Euler--Lagrange equations associated
with the energy functional $\mathcal{E}(\rho;\, \rho_\mathrm{gs})$.
This gives rise to the Kohn--Sham-like equations
\begin{equation}
    \label{eq:kseqns}
    \qty(- \frac{1}{2} \nabla^2 + v_i^{\eps}(\mathbf{r})) \varphi_{i}^{\eps,j}(\rr) = e_{i}^{\eps,j}\varphi_{i}^{\eps,j} (\rr)
    \quad \text{where} \quad
    v_{i}^\eps = \alpha v_\mathrm{H}(\rho^\varepsilon_{i-1}) + \beta v_\mathrm{H}(\rho_\mathrm{gs})  + \gamma v_\mathrm{ext} +  \frac{1}{\eps}J(\rho^\varepsilon_{i-1} - \rho_\mathrm{gs}).
\end{equation}
For each iterate $i\geq 1$, the proximal density
is computed as $\rho^\varepsilon_{i}(\rr) = \sum_{j=1}^N |\varphi_{i}^{\eps,j}(\rr)|^2$.
The $e_{i}^{\eps,1} \leq \dots \leq e_{i}^{\eps,N}$ are the Lagrange
multipliers associated with the orthonormality constraint. After
transforming $\varphi_{i}^{\eps,j}$ to the eigenbasis of $\hat{t} +  v_i^\eps$,
the Lagrange multipliers are the eigenvalues of the KS equations, i.e., describe the orbital energies.
If the SCF sequence $\rho^\varepsilon_i$ converges for $i\to \infty$,
then $\rho^\varepsilon_i \to \rho^\eps$ since the proximal point $\rho^\eps$ is a unique
point (Proposition~\ref{prop:UniqueProx}).
We can therefore start from some initial guess $\rho^\varepsilon_0 = \sum_{j=1}^N |\varphi_0^{\eps,j}|^2$,
and iterate \cref{eq:kseqns} in order to find the proximal point for a given $\eps$.

\subsubsection{A Proximal Point Algorithm}\label{subsec:ProxPointAlg}
In general, algorithms that compute the proximal point tackle the minimisation of proper, convex, and l.s.c.\  functionals. The proximal density, $\rho^\eps = \Pi_\mathcal{F_{\alpha,\beta,\gamma}}^\eps(\rho_\mathrm{gs})$, is the minimiser of $\mathcal{E}(\rho;\,\rho_\mathrm{gs})$ from \cref{eq:MYEnergyFunctional}, which in the $\Xdens$-topology is proper, convex, and l.s.c.
Iteratively determining the proximal point for a given sequence $\{\lambda_k\}_{k=0}^\infty \subset \mathbb{R}_+$ such that $\sum_{k=0}^\infty \lambda_k = +\infty$ and a starting guess $\rho^\varepsilon_0 \in \Xdensaff$, 
\begin{equation}\label{eq:ProximalIteration}
    \rho_{k+1}^\eps = \Pi^{\lambda_k}_{\mathcal{E}(\cdot;\, \rho_\mathrm{gs})}(\rho_k^\eps) 
    = \argmin_{\sigma \in \Xdensaff} \qty{\mathcal{E}(\sigma;\,\rho_\mathrm{gs}) + \frac{1}{2\lambda_k}\norm{\sigma-\rho^\varepsilon_k}^2_{\Xdens}} 
    \quad \text{for} \quad k \in \mathbb{N},
\end{equation}
is then a proximal point algorithm for minimising $\mathcal{E}(\rho;\,\rho_\mathrm{gs})$. In particular, it holds that $\rho^\varepsilon_k \rightharpoonup \rho^\eps$ (weakly) in $\Xdens$ as $k\to \infty$~\cite[Thm.~27.1]{Bauschke_2017} and additionally that\footnote{By rewriting \cite[Eq.~(27.5)]{Bauschke_2017} as  $2 \gamma_n (f(x_{n+1}) - \inf f) \leq \norm{x_n - z}^2 -\norm{x_{n+1} - z}^2$ (with $z$ as the minimiser of $f$), summing from $n=0$ to $k$ over both sides and using that $f(x_{n+1}) \leq f(x_n)$ for all $n$, the stated result follows.}
\begin{equation} \label{eq:ProxAlgEnergyConv}
    \mathcal{E}(\rho^\varepsilon_{k+1};\rho_\mathrm{gs}) -\mathcal{E}(\rho^\eps;\, \rho_\mathrm{gs}) \leq  \frac{\norm{\rho^\varepsilon_0 - \rho^\eps}^2_{\Xdens}}{2\sum^{k}_{i=0}\lambda_i} \to  0^+.
\end{equation}
Moreover, the accuracy of the energy functional $\mathcal{E}(\rho^\varepsilon_{k};\, \rho_\mathrm{gs})$ is bounded by the accuracy of the starting guess $\rho^\varepsilon_0$ and the rate of divergence of $\sum_{i=0}^k \lambda_i$ with respect to $k$. The parameters $\lambda_k>0$ may \textit{a priori}  be chosen arbitrarily. However, if all $\lambda_k$ are identical, then $\mathcal{E}(\rho^\varepsilon_{k};\rho_\mathrm{gs})$ converges to $ \mathcal{E}(\rho^\eps;\, \rho_\mathrm{gs})$ as $1/k$ and the same rate of convergence also holds if all $\lambda_k$ are uniformly bounded away from zero. Note that \cref{eq:ProxAlgEnergyConv} does not automatically give any estimate for the rate of convergence of $\rho^\eps \to \rho_\mathrm{gs}$.

From \cref{eq:ProximalIteration}, we recognise that the proximal iteration satisfies the stationarity condition 
\begin{equation*}
    \underline{\partial}\mathcal{E}(\rho^\varepsilon_{k+1};\, \rho_\mathrm{gs}) + \frac{1}{\lambda_k} J\qty(\rho^\varepsilon_{k+1} - \rho^\varepsilon_k) \ni 0.
\end{equation*}
For the choice of guiding functional $\mathcal{F}_{\alpha,\beta,\gamma}$ from \cref{eq:variableGuideFunc} this implies that (using $v_\mathrm{H}(\cdot) = J(\cdot)$ and the linearity of $J$ for $\Xdensaff$)
\begin{equation*}
\begin{aligned}
    -\underline{\partial}T_\mathrm{DM}(\rho^\varepsilon_{k+1}) &\ni  \alpha v_\mathrm{H}(\rho^\varepsilon_{k+1}) + \beta v_\mathrm{H}(\rho_\mathrm{gs}) + \gamma v_\mathrm{ext}  + \frac{1}{\eps}J\qty(\rho^\varepsilon_{k+1}- \rho_\mathrm{gs}) + \frac{1}{\lambda_k} J\qty(\rho^\varepsilon_{k+1} - \rho^\varepsilon_k)\\
    &=\alpha v_\mathrm{H}(\rho^\varepsilon_{k}) + \beta v_\mathrm{H}(\rho_\mathrm{gs}) + \gamma v_\mathrm{ext}  + \frac{1}{\eps}J\qty(\rho^\varepsilon_{k}- \rho_\mathrm{gs}) + \frac{\eps + \lambda_k+ \eps\lambda_k \alpha}{\eps\lambda_k} J\qty(\rho^\varepsilon_{k+1} - \rho^\varepsilon_k).
\end{aligned}
\end{equation*}
The proximal point algorithm thus gives rise to a KS-type potential depending on the previous step $\rho^\varepsilon_{k}$ equivalent to \cref{eq:kseqns} with an additional term due to the difference between the iterates. In particular, the stationarity condition gives rise to the equations
\begin{equation}\label{eq:ProxKSeqs}
    \qty[- \frac{1}{2} \nabla^2 +\alpha v_\mathrm{H}(\rho^\varepsilon_{k})+ \beta v_\mathrm{H}(\rho_\mathrm{gs}) + \gamma v_\mathrm{ext}   + \frac{1}{\eps}J\qty(\rho^\varepsilon_{k}-\rho_\mathrm{gs}) + \frac{\eps+\lambda_k+\eps\lambda_k \alpha}{\eps\lambda_k}J\qty(\rho^\varepsilon_{k+1} - \rho^\varepsilon_k)] \varphi_{k+1}^{\eps,j} = e_{k+1}^{\eps,j}\varphi_{k+1} ^{\eps,j},
\end{equation}
where $\rho^\varepsilon_{k+1}(\rr) = \sum_{j=1}^N  |\varphi_{k+1}^{\eps,j}(\rr)|^2$. Note that \cref{eq:ProxKSeqs} includes the orbitals $\varphi_{k+1}^{\eps,j} (\rr)$ through $\rho^\varepsilon_{k+1}$ in a non-linear way. Thus, solving \cref{eq:ProxKSeqs} would introduce an SCF problem for each iterate, where at self-consistency $e_{k+1}^{\eps,1} \leq\cdots \leq e_{k+1}^{\eps,N}$ are the orbital energies and all orbitals $\varphi_{k+1}^{\eps,j} (\rr)$ are (singly) occupied, as in  \cref{subsec:SCFApproach}. In practice, this limits the usability of the method significantly as it introduces an additional step that must be solved self-consistently.

For our numerical investigations on the periodic Gross--Pitaevskii equation, the modified energy functional $\mathcal{E}(\sigma;\,\rho_\mathrm{gs}) + \frac{1}{2\lambda_k}\norm{\sigma-\rho^\varepsilon_k}^2_{\Xdens}$ is instead minimised directly over $\sigma\in \Xdensaff$ or gets rewritten into a partial differential equation. This will be addressed further in \cref{sec:gpe:discretisation}.

\subsection{The Periodic Gross--Pitaevskii equation}\label{subsec:GPE}
To explore the inversion algorithm, let us consider the one-dimensional \textit{periodic Gross--Pitaevskii equation} (GPE) on the interval $\Omega = [0,a]$ with periodic boundary conditions, defined by the non-linear one-particle Schrödinger equation
\begin{equation}\label{eq:GPEquation}
    \qty[-\frac{1}{2}\nabla^2 + v + g\rho] \varphi = E \varphi,
    \quad \text{where} \quad \rho = |\varphi|^2 .
\end{equation}
Here $v$ is a periodic potential and $g>0$  is the coupling constant. We define the Gross--Pitaevskii energy functional $E(v)$ as the ground-state energy for a given potential $v\in\Xpot$,
\begin{equation*}
    E(v) = \inf_{\rho \in \mathcal{I}_1(\Omega)} \qty{\frac{1}{2}\|\nabla\sqrt{\rho}\|^2_{L^2} + \langle v, \rho \rangle + \frac{g}{2}\|\rho\|^2_{L^2}}.
\end{equation*}
For the infimum, one easily verifies that \cref{eq:GPEquation} is the corresponding first-order optimality condition. 
To motivate our use of the GPE, we want to highlight the similarity of the KS equations (\cref{eq:kseqns}) to \cref{eq:GPEquation}. Both are non-linear eigenvalue problems, since the potentials $v + g\rho$ in \cref{eq:GPEquation} and $v^{\eps}_i$ in \cref{eq:kseqns} depend on
the density, and therefore on the orbital-like quantities $\varphi$ and $\varphi_i^{\eps,j}$, respectively. By analogy, the GPE ground state $\varphi$ is also called an orbital.
Notably, the GPE problem is simpler as it 
requires the solution of only a single eigenpair $(E, \varphi)$, where in contrast, the KS equations need to be solved for
$N$ such eigenpairs $(e_i^{\eps,j}, \varphi_i^{\eps,j})$
to make up the next density $\rho_i^\eps$. The simplified setting of the GPE enables the numerical study of the algorithms presented in \cref{subsec:InversionApproaches}.

Since the GPE is a one-particle model,
the KS-like potential is trivially $v_\mathrm{KS} = v_\mathrm{ext} + g\rho$. In this setting, knowledge of $\rho_\mathrm{gs}$ directly determines the corresponding potential since the sought potential is $v_g(\rho_\mathrm{gs}) = g \rho_\mathrm{gs} $. 
The goal of the KS inversion is nevertheless to obtain $v_g \in \Xpot$ using the Moreau--Yosida-based inversion scheme outlined in \cref{subsec:InversionAlgorithm} above. The guiding functional
\begin{equation*}
    \mathcal{F}(\rho) = T(\rho) + \langle v_\mathrm{ext},  \rho \rangle
\end{equation*}
satisfies Assumption~\ref{assump:FX} and is applicable to the Moreau--Yosida regularisation. Given a ground-state density $\rho_\mathrm{gs} \in \Xdensaff$, the corresponding proximal density $\rho^\eps = \Pi_\mathcal{F}^\eps(\rho_\mathrm{gs})$ is the unique minimiser of the energy functional
\begin{equation} \label{eq:GPEnergyFunc}
    \mathcal{E}(\rho;\rho_\mathrm{gs}) = T(\rho) + \langle v_\mathrm{ext}, \rho\rangle + \frac{1}{2\eps} \norm{\rho - \rho_\mathrm{gs}}_{\Xdens}^2.
\end{equation}
It then follows from Corollary~\ref{cor:DensityPotentialIvesion} that the potential determined by the inversion scheme is 
\begin{equation} \label{eq:GPsequence}
    v_g = \lim_{\eps \to 0^+} v_g^\eps 
        = \lim_{\eps \to 0^+} \frac{1}{\eps}J(\rho^\eps - \rho_\mathrm{gs})
        = J\qty(\left. \partial^+_\eps\rho^\eps\right|_{\eps=0}).
\end{equation}

\subsubsection{Computational Setup}
\label{sec:gpe:discretisation}
The main goal of the implementation of the inversion scheme is to obtain an accurate numerical representation of the potential, $v_g$, given some reference ground-state density, $\rho_\mathrm{gs}$. The potential is obtained by extrapolating the limit $\eps\to 0^+$ in \cref{eq:GPsequence} for an exponentially decreasing sequence of $\eps$.
To calculate $v^\varepsilon_g$ for each $\eps$ we utilise a
plane-wave discretisation \cref{eq:FourierRepresentation}
for all involved quantities, such as the orbital $\varphi$, the densities $\rho$ and $\rho_\mathrm{gs}$
and the potentials $v^\varepsilon_g$.
The set of employed $e_\mathbf{G}$ is limited by introducing a finite cut-off energy $E_\mathrm{cut}$
and restricting the norm of the employed wave vector $\mathbf{G}$ to a finite value.
In particular, we restrict the basis for $\varphi$ such that
$\norm{\mathbf{G}}_2 \leq \sqrt{2 E_\mathrm{cut}}$ and the basis for all densities and potentials
such that $\norm{\mathbf{G}}_\infty \leq 2 \sqrt{2 E_{\mathrm{cut}}}$.
The basis for $\rho$, $\rho_\mathrm{gs}$ and $v^\varepsilon_g$ is thus roughly twice as large (in 1D)
as the basis for $\varphi$.
Employing these two different discretisations is standard for plane-wave DFT and ensures
that no additional numerical error
is introduced when computing the discretised density from the discretised orbitals with $\rho_\varphi(r) = |\varphi(r)|^2$.
Moreover, the use of the infinity norm for defining the density and potential basis sets leads
to a Cartesian grid in the Fourier frequencies,
ensuring that fast-Fourier transforms can be employed efficiently
when computing convolutions or point-wise operations.
See \cite[Sec.~2.3]{HerbstSun2025} for more details on the setup of discretisation grids for plane-wave DFT.

\textbf{Direct minimisation and SCF.}
In analogy to \cref{subsec:DirectMinimisation} the energy functional of \cref{eq:GPEnergyFunc} can be rewritten as the orbital minimisation problem
\begin{equation*}
    \mathcal{E}(\varphi;\rho_\mathrm{gs})  = \frac{1}{2}\int_{\Omega} \abs{\nabla \varphi(r)}^2 \dd{r}  + \int_\Omega v_\mathrm{ext}(r)  \rho_\varphi(r) \dd{r} + \frac{1}{2\eps} \norm{\rho_\varphi - \rho_\mathrm{gs}}_{\Xdens}^2.
\end{equation*}
Since $\mathcal{E}(\varphi;\rho_\mathrm{gs})$ has a similar structure as the usual energy expressions found in KS-DFT, standard minimisation techniques from KS-DFT apply. In particular, in anticipation of our treatment of bulk materials,
we minimise $\mathcal{E}$ using a BFGS-based quasi-Newton scheme implemented in DFTK~\cite{DFTKpaper},
where the proximal density and corresponding orbital is used as a starting guess for the next iteration in $\eps$. This approach will be referred to as the direct minimisation approach. To test convergence we will also consider an extension employing a full Newton scheme. Furthermore, we implemented an SCF-based computation of the proximal point,
where we have used a damped fixed-point solver for the SCF cycles. This will be referred to as the SCF implementation.

\textbf{PDE solver approach.}
In order to verify the accuracy of the inversion scheme, we make use of the fact that the Gross--Pitaevskii equation is a simple one-dimensional model. Since the proximal density is a stationary solution of $\underline{\partial}\mathcal{E}(\rho;\,\rho_\mathrm{gs})$ and that $\varphi(r) = \sqrt{\rho(r)}$, the proximal density may be obtained as a solution to the differential equation 
\begin{equation*}
    -\frac{1}{4} \frac{\nabla^2 \rho}{\rho}
    + \frac{1}{8} \frac{\left| \nabla \rho \right|^2}{\rho^2}
    + v_{\mathrm{ext}}
    + \frac{1}{\eps} J(\rho - \rho_{\mathrm{gs}})  
    + \mu = 0,
\end{equation*}
where $\mu$ is the Lagrange multiplier associated with the normalisation condition, i.e., $\int_\Omega\rho(r) \dd{r} = 1$. 
In our implementation, the proximal density is determined using a standard Newton-based partial differential equation (PDE) solver, i.e., from the Julia library \texttt{Optim}~\cite{Mogensen2018}. The problem is discretised using the same Fourier cut-off as described above, corresponding to a uniform one-dimensional spatial grid. The differential operators are evaluated in real space using first- and second-order finite differences, with periodic boundary conditions. Note here that DFTK also utilises \texttt{Optim} for its quasi-Newton implementation of the direct minimisation approach. 

\textbf{Proximal point algorithm.} Using a similar strategy to the minimisation of $\mathcal{E}(\varphi;\rho_\mathrm{gs})$, we may also investigate the proximal point algorithm presented in \cref{subsec:ProxPointAlg}. In particular, the minimisation of the modified energy functional $\mathcal{E}(\varphi;\, \rho_\mathrm{gs}) + \frac{1}{2\lambda_k} \norm{\rho_\varphi - \rho_k^\eps}_{\Xdensaff}^2$ can be used to determine the proximal point $\rho^\varepsilon_{k+1}$ iteratively from some staring guess $\rho^\varepsilon_0$. To investigate the accuracy and efficiency of the proximal algorithm, the proximal points $\rho^\varepsilon_{k}$ are determined by numerically solving the differential equation arising from the corresponding modified energy functional using the aforementioned PDE solver.

In our numerical setup, the external potential is chosen to be an optical lattice potential, $v_\mathrm{ext}(r) =-\sin^2(\pi r/a)$, and the discretisation is dictated by $E_\mathrm{cut} = 3000$~Ha as described above. The reference calculation is performed using the SCF approach where the density is converged such that the $L^2$ difference between iterates is $\sim10^{-14}$. Moreover, for the inversion we use the same discretisation setting and we employ the heuristic stopping criterion between iterates $k$ of $\norm{\rho^\varepsilon_k - \rho^\varepsilon_{k-1}}_\Xdens < \delta \cdot \eps $ introduced in \cite{Herbst_2025}. All results presented in this section are calculated with $\delta=10^{-3}$ with a cell length of $a=10$ and unit coupling $g=1$. 

\subsubsection{Results}
\Cref{fig:GPProximalConvergence} shows the convergence of the proximal density $\rho^\varepsilon$ and the determined potential $v_g^\varepsilon$ to $\rho_\mathrm{gs}$ and $v_g$, respectively. The inversion was performed using five different methods for computing the proximal density. Apart from the inversion using an SCF approach, the convergence of the proximal density displays no significant discrepancies between the methods, however, the potential recovered by the inversion is more sensitive to the method of choice. For values of $\varepsilon\lesssim 10^{-2}$ the SCF approach fails to reach the stopping criterion, and as a result, the inversion is terminated at $\varepsilon=10^{-4}$. Similarly, for values of $\varepsilon\lesssim 10^{-4}$, the direct minimisation method (the quasi-Newton) also fails to converge, and despite that the fact that the proximal density continues to approach the ground-state density, the potential does not improve. The Newton algorithm implemented in DFTK reaches its convergence criterion for all $\varepsilon$ values displayed in \cref{fig:GPProximalConvergence}, and fully agrees with the results obtained from the PDE approach. 
To separate the proximal point algorithm from the usual quantum-chemistry methods it is implemented using the PDE solver. Thus, the proximal point algorithm exhibits the same density and potential convergence as the PDE solver.

\begin{figure}
    \centering
    \includegraphics[width=\linewidth]{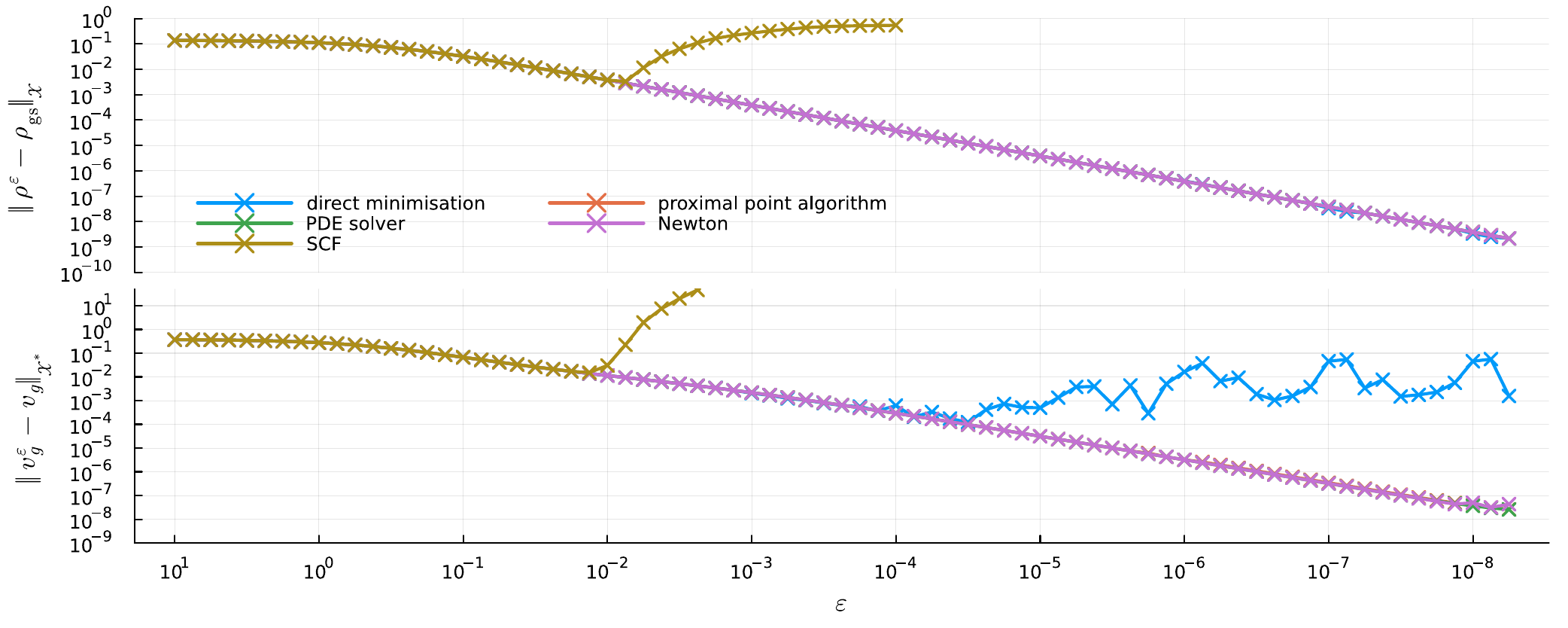}
    \caption{The convergence of the inversion scheme for different methods of computing the proximal density. (top) The convergence of $\rho^\varepsilon$ to the ground-state density $\rho_\mathrm{gs}$ measured in the norm of $\Xdens$ for an exponentially decreasing sequence in $\varepsilon$. (bottom) The convergence of the potential $v_g^\varepsilon$ determined by the inversion to the exact reference $v_g$, measured in the $\Xpot$-norm.}
    \label{fig:GPProximalConvergence}
\end{figure}

\begin{figure}
    \centering
    \includegraphics[width=\linewidth]{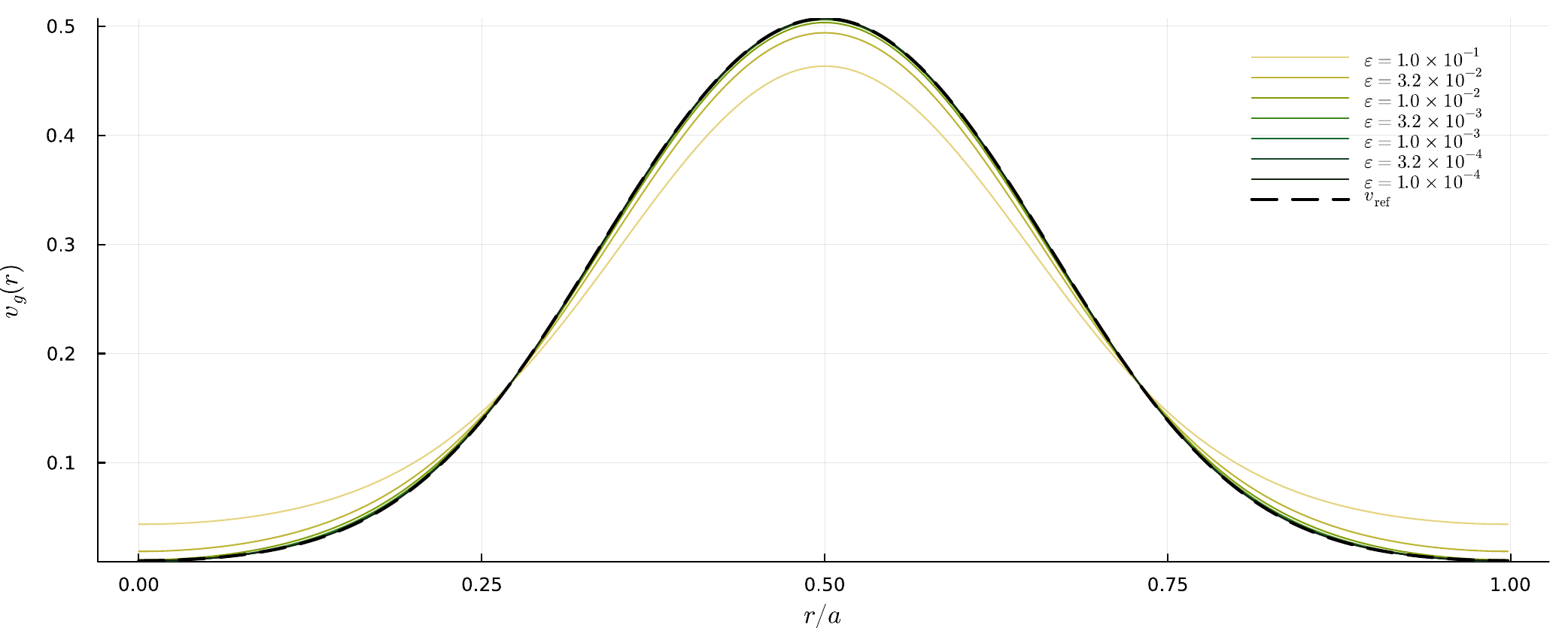}
    \caption{Real space plot of the potentials $v^\varepsilon_g(r)$ obtained by the inversion scheme for select values of $\varepsilon$ obtained using DFTK's Newton implementation. The true potential $v_g(r) = g\rho(r)$ is shown as black dashed line for reference.}
    \label{fig:GPpotential}
\end{figure}

To illustrate the density-potential inversion scheme, \cref{fig:GPpotential} displays the real-space plot of the potentials $v_g^\eps$ obtained by the inversion for select values of $\eps$. The exact reference potential $v_g(r) = g\rho(r)$ is shown for comparison. Since any potential in $\Xpot$ has zero mean, the displayed potentials $v_g^\eps$ have been shifted to the mean value of $v_g$. Moreover, the potential was obtained by determining the proximal density using the implementation of direct minimisation with a full Newton method.

\begin{figure}
    \centering
    \includegraphics[width=\linewidth]{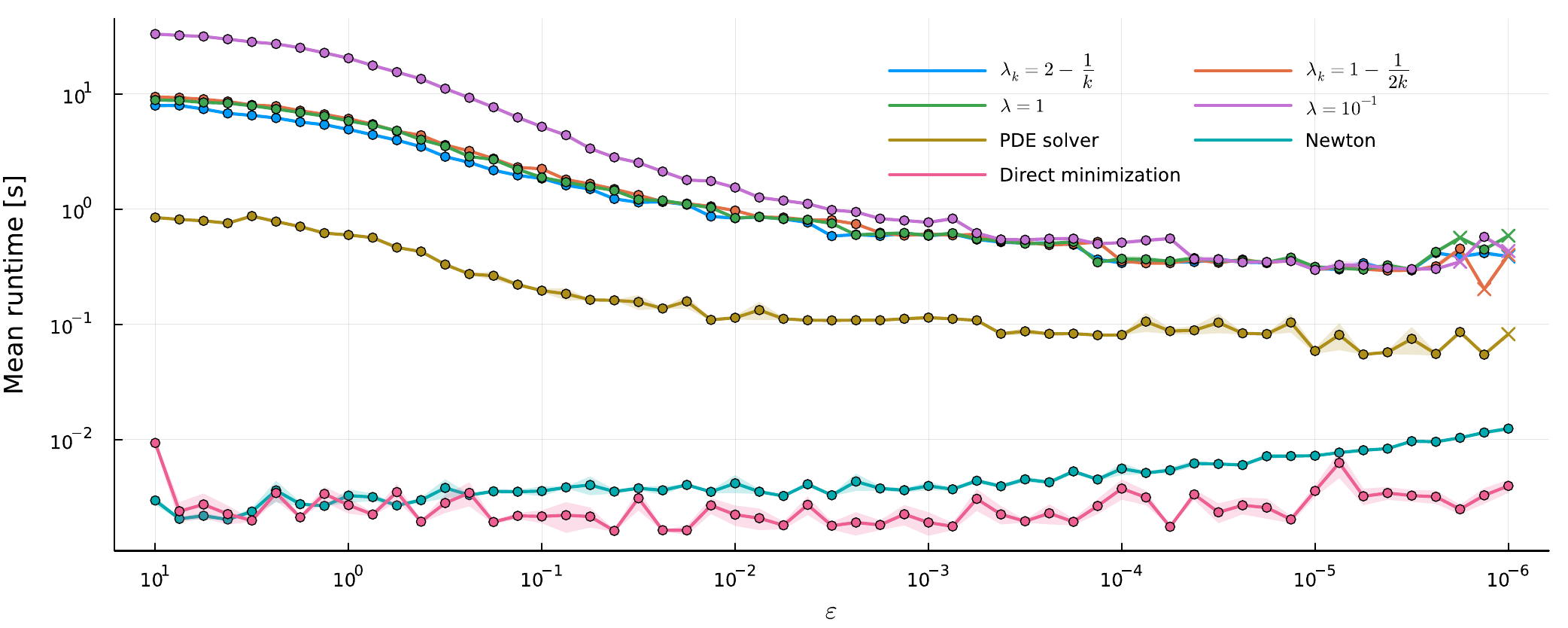}
    \caption{
    The mean computational time for the proximal point algorithm given in \cref{eq:ProximalIteration} in order to reach the desired tolerance of $\norm{\rho^\varepsilon_k - \rho^\varepsilon_{k-1}}_{\Xdens} \leq \eps \cdot 10^{-3}$ at a given value of $\eps$ when the external potential is $v_\mathrm{ext}(r) = -\sin^2(\pi r/a)$ and the $E_\mathrm{cut}=3000$. All data points drawn with a dot are converged  the prescribed tolerance, whereas values marked with a cross are failed to reach the tolerance. For comparison, the PDE solver, the quasi-Newton and full Newton methods for direct minimisation are also shown. For each method the runtime is averaged over five runs and the standard error from the mean is drawn as a shaded area.}
    \label{fig:GPfcalls}
\end{figure}

To illustrate the proximal point algorithm introduced in \cref{subsec:ProxPointAlg}, we examine its numerical behaviour for several choices of $\lambda$-sequences and compare to other methods discussed above.  Although the methods differ structurally, the computational time required to determine the proximal density for each value of $\varepsilon$ provides a useful estimate of the overall computational cost. 
There is considerable flexibility in the choice of the sequence $\{\lambda_k\}$ in the proximal point algorithm. Recall that if the $\lambda_k$ are uniformly bounded away from zero, then convergence of the energy functional is guaranteed. We therefore investigate the computational cost of achieving convergence for several representative $\lambda$-sequences. 
\Cref{fig:GPfcalls} shows the mean computational time (in seconds) required to meet the prescribed convergence criterion for each value of $\varepsilon$ averaged over five runs, and the shaded area indicates the standard error from the mean. Data points marked with a dot indicate successful convergence, whereas crosses denote parameter values for which the method failed to converge for at least one run.
Since the direct minimisation approaches (quasi-Newton and Newton) are optimised for the periodic setting used in DFTK, they exhibit faster convergence than the unspecialised PDE solver employed here. Nevertheless, the proximal point algorithm (i.e., using some $\lambda$-sequence) incurs a significant computational overhead relative to the method used to compute the proximal density since it entails additional proximal point calculations.

\subsection{Bulk Materials}\label{subsec:BulkMaterials}
The periodic homogeneous Sobolev setting introduced in \cref{subsec:homogeneousSobolev} is particularly suited for density-potential inversion on bulk materials as the $\Xdensaff$-norm aligns with the mean-field interaction energy (the Hartree energy) as stated in Proposition~\ref{prop:EH}. In this section, we numerically investigate the KS inversion scheme summarised in \cref{subsec:PeriodicIKS} on densities for selected bulk materials. The reference densities are obtained from ground-state KS calculations using the SCF approach described above. In this work, we address the inversion on two semiconductors (silicon and gallium arsenide) and two ionic salts (sodium chloride and potassium chloride). 

\subsubsection{Direct Minimisation for Bulk Materials}
In condensed matter physics bulk materials are usually modelled as perfectly periodic crystals.
As a result the unit cell $\Omega$ repeats infinitely along the lattice $\mathcal{R}$.
To avoid dealing with infinite repetitions, the standard approximation is the supercell approach,
where the unit  cell $\Omega$ is only repeated a finite number of times in each direction.
For example, an $l_1 \times l_2 \times l_3$ supercell is formed by repeating $\Omega$
respectively $l_1$, $l_2$ and $l_3$ times along the respective lattice directions.
At the boundary of this larger simulation supercell $\Omega^L$ one commonly
still employs periodic boundary conditions to avoid artificial surface effects.
This matches the setting of the periodic homogeneous Sobolev spaces introduced
in \cref{subsec:homogeneousSobolev} with the difference that the simulation box $\Omega^L$
is now $L = l_1 l_2 l_3$ times larger.
In particular, we emphasise that the orbitals $\varphi^j$,
which are introduced in \cref{subsec:DirectMinimisation,subsec:SCFApproach} 
to numerically solve the SCF or direct minimisation problem,
are now entities from $H^1_\mathrm{per}(\Omega^L)$---thus only periodic with respect to the supercell and \emph{not} with respect to $\mathcal{R}$. A naive numerical treatment of e.g.~the direct minimisation problem of \cref{eq:dmprob} thus requires a discretisation within the much larger simulation box $\Omega^L$.

However, in the absence of spontaneous symmetry breaking the ground-state density
$\rho_\mathrm{gs}$ keeps the $\mathcal{R}$-translational symmetry in bulk materials.
Since this implies that that the $\eps \to 0^+$ limit
of the proximal densities $\rho^\eps$ is $\mathcal{R}$-periodic,
it is natural to assume the $\mathcal{R}$-periodicity of $\rho^\eps$
for each $\eps$. To ensure this periodicity during the
minimisation problem \cref{eq:dmprob} we employ a Bloch transformation:
we relabel the supercell orbitals
$\varphi^j$ as $\varphi^{j,\mathbf{k}}(\mathbf{r}) = \rme^{\rmi\mathbf{k}\cdot{\mathbf{r}}} u^{j,\mathbf{k}}(\mathbf{r})$
and rewrite the density as
\begin{equation*}
     \rho_\Phi(\mathbf{r}) 
    = \dfrac{1}{L} \sum_{\mathbf{k} \in \mpgrid} \sum_{j=1}^{N} |\varphi^{j,\mathbf{k}}(\mathbf{r})|^2
    = \dfrac{1}{L} \sum_{\mathbf{k} \in \mpgrid} \sum_{j=1}^{N} |u^{j,\mathbf{k}}(\bm r)|^2.
\end{equation*}
   
In these expressions the $u^{j,\mathbf{k}}(\mathbf{r}) \in H^1_\mathrm{per}(\Omega)$
are again $\mathcal{R}$-periodic functions
and the $\mathbf{k}$-points are taken from the Monkhorst--Pack grid~\cite{monkhorst1976special}
$\mpgrid = \Omega^* \cap (
    \mathbb{Z}\mathbf{b}_1/l_1 + \mathbb{Z}\mathbf{b}_2/l_2 + \mathbb{Z}\mathbf{b}_3/l_3
)$,
where $\mathbf{b}_1$, $\mathbf{b}_2$, $\mathbf{b}_3$ are the (reciprocal) unit cell vectors spanning
the first Brillouin zone $\Omega^*$, the unit cell of the dual lattice $\mathcal{R}^\ast$.
Indeed, one easily verifies the resulting density to remain $\mathcal{R}$-periodic.
Inserting the Bloch factorisation into \cref{eq:dmprob} finally leads to the equivalent minimisation problem
\begin{equation} \label{eqn:kdmprob}
    \min\qty{
        \frac{1}{L}
        \sum_{\mathbf{k}\in\mpgrid}
        \sum_{n=1}^N 
        \frac{1}{2}{}\langle u^{n,\mathbf{k}}, (-i \nabla + \mathbf{k})^2 u^{n,\mathbf{k}} \rangle
        + \mathcal{G}_{\alpha,\beta,\gamma}^\eps(\rho_\Phi; \rho_\mathrm{gs}) \ : \
        \begin{matrix}
              u^{1,\mathbf{k}}, \ldots, u^{N,\mathbf{k}} \in L^2_\mathrm{per}(\Omega),\; & \forall {\mathbf{k}} \in \mpgrid,  \\
              \langle u^{n,\mathbf{k}}, u^{m,\mathbf{k}}\rangle = \delta_{n,m}, & \forall n,m= 1,\dots, N
        \end{matrix}}.
\end{equation}
This is the equation we consider when employing the direct minimisation approach
for computing the proximal point. Notably all its unknowns ($u^{n,\mathbf{k}}$, respectively $\rho_\Phi$)
are $\mathcal{R}$-periodic, such that again a discretisation within only the unit cell $\Omega$
is sufficient to solve this problem numerically.
Instead of setting the number of supercell repeats $l_1$, $l_2$, $l_3$ one often determines
these values implicitly by determining a desired maximal spacing of the $\mathbf{k}$-points
in $\mpgrid$.
For more details on the supercell approach in the context of forward KS-DFT,
see~\cite[Sec.~2.8]{lin2019math_electronic} and for more details on Bloch theory
we refer the reader to~\cite[Ch.~7]{Lewin_2024}.

Apart from considering the simulation of periodic supercells,
practical bulk material simulations commonly feature the use of pseudopotentials.
Compared to employing a bare Coulomb interaction in the potential $v_\mathrm{ext}$
for describing the attraction of electrons to nuclei, these pseudopotentials
are much smoother, such that their Fourier-space representation features faster decay,
which in turn implies that smaller plane-wave basis sets are required to reach convergence.
Beyond modifying the electron-nuclear interactions, the introduction of pseudopotentials
in a forward  KS calculation implies an additional change to the minimised
energy functional, namely the addition of a non-local Kleinman--Bylander (KB) term,
an operator with a low-rank integral kernel
$
V_\mathrm{KB}(\mathbf{r}, \mathbf{r}') = \sum_{i} d_i\, b_{i}(\mathbf{r})\, \overline{b_{i}(\mathbf{r}')},
$
where $d_i \in \mathbb{R}$ and the index $i$ labels the projectors $b_i \in L^2_\mathrm{per}(\Omega)$
of which there are a few tens per atom in the unit cell.
While our theory is not yet able to accommodate this term rigorously
(since we assume a strictly local KS potential),
the KB term is needed to numerically treat realistic applications in plane-wave discretisations.
For our inversion on bulk materials we will therefore add a term
$\langle \varphi^{n,\mathbf{k}}, V_\mathrm{KB} \varphi^{m,\mathbf{k}} \rangle$
to the objective function of the minimisation problem \cref{eqn:kdmprob}
corresponding to the 
pseudopotential, which was used for generating
the reference density $\rho_\mathrm{gs}$.

\subsubsection{Computational Setup}
In line with standard practice in KS-DFT calculations using plane-wave basis sets
we employ a similar setup to \cref{sec:gpe:discretisation} for discretising \cref{eqn:kdmprob}:
based on a cut-off parameter $E_\mathrm{cut}$ we construct a basis for the densities
by restricting Fourier wave vectors according to $\norm{\mathbf{G}}_\infty \leq 2 \sqrt{2 E_{\mathrm{cut}}}$,
while for the periodic part of the Bloch orbitals $u_{n,\mathbf{k}}$ the corresponding
discretisation basis is constructed according to $\norm{\mathbf{G} + \mathbf{k}}_2 \leq \sqrt{2 E_\mathrm{cut}}$.
Notably, on each $\mathbf{k}$-point a different plane-wave basis is thus employed to discretise
the corresponding orbitals.
Within this discretised setting, our results are obtained following the same schematic setup as presented in \cite{Herbst_2025}: First, we perform a forward KS computation using an SCF approach implemented in DFTK~\cite{DFTKpaper}, yielding the reference ground-state density $\rho_\mathrm{gs}$ required for our inversion. In this reference calculation, we use the PBE xc functional~\cite{Perdew1996}  together with the corresponding PBE PseudoDojo pseudopotentials~\cite{VanSetten2018}. The reference xc potential $v_\mathrm{xc}$ is then obtained self-consistently within this forward KS calculation and is exact at the level of the density-functional approximation with the aforementioned corrections.
For all systems, we use a  $\mathbf{k}$-point spacing of at most $0.12\, \text{\r{A}}{}^{-1}$
to construct the Monkhorst--Pack grid $\mpgrid$
and the energy cut-off is taken to be roughly twice the recommended values for the pseudopotentials used. For the four bulk materials of interest we use: for bulk silicon (Si) $E_\mathrm{cut} = 45$~Ha, for potassium chloride (KCl) $E_\mathrm{cut} = 65$~Ha, for gallium arsenide (GaAs) $E_\mathrm{cut} = 83$~Ha, and for sodium chloride (NaCl) $E_\mathrm{cut} = 83$~Ha.

Next, using the density $\rho_\mathrm{gs}$ obtained in the forward KS calculation together with the guiding functional $\mathcal{F}$ (defined in \cref{eq:GuidingFunctional}), where the kinetic energy contribution is approximated by $T_\mathrm{S}$, the KS inversion is performed in the same discretisation setting, i.e., $\mathbf{k}$-point spacing and energy cut-off.   
The proximal density $\rho^\varepsilon = \Pi^\varepsilon_{\mathcal{F}}(\rho_\mathrm{gs})$ is computed for an exponentially decreasing sequence of $\eps$, and the corresponding potential is extracted as $v_\mathrm{xc}^\eps = \tfrac{1}{\eps} J(\rho^\eps - \rho_\mathrm{gs})$. The proximal density is determined using of the direct minimisation approach described in \cref{subsec:DirectMinimisation}. The corresponding orbital minimisation problem (\cref{eq:dmprob}) is solved using a BFGS-based quasi-Newton scheme~\cite{Edelman1998,Boumal2023} that has been adapted to the geometrical structure of the Stiefel manifold arising from the orthogonality constraint on the orbitals. For each value of $\eps$, the direct minimisation is terminated using the heuristic stopping criterion introduced in \cite{Herbst_2025}, namely $\| \rho^\varepsilon_k - \rho^\varepsilon_{k-1}\|_\Xdens < \delta \cdot \eps$, where $\rho^\varepsilon_k$ denotes density of iterate $k$. All inversion results for the bulk materials use $\delta = 10^{-2}$.
In the event that the direct minimisation approach using the quasi-Newton method fails to converge, the inversion reverts to a full Newton method. However, efficient convergence of the Newton method typically requires an initial guess sufficiently close to the solution. Consequently, the direct minimisation using the full Newton method is always initialised using the orbitals produced by (possibly unconverged) quasi-Newton method. Because the full Newton method is significantly more computationally demanding, employing it only when the quasi-Newton method fails to converge provides an effective balance between cost and robustness. 

Following \citet{Bohle2026}, we numerically investigate the role of 
the ($\alpha,\beta,\gamma$)-dependent guiding functional in \cref{eq:variableGuideFunc} for its performance in the inversion procedure. In this setting, the reference potential $v_\mathrm{xc}$ obtained in the forward KS calculation is appropriately shifted to match the potential determined by the inversion. Moreover, note that when tuning the amount of external potential used in the guiding functional, as parametrised by $\gamma$, we scale only the local atomic potential (as no other local ``external'' potentials are present). The pseudopotential corrections as well as the non-local part of the atomic potential are then left unchanged

\subsubsection{Results} 
\begin{figure}
    \centering
    \includegraphics[width=\linewidth]{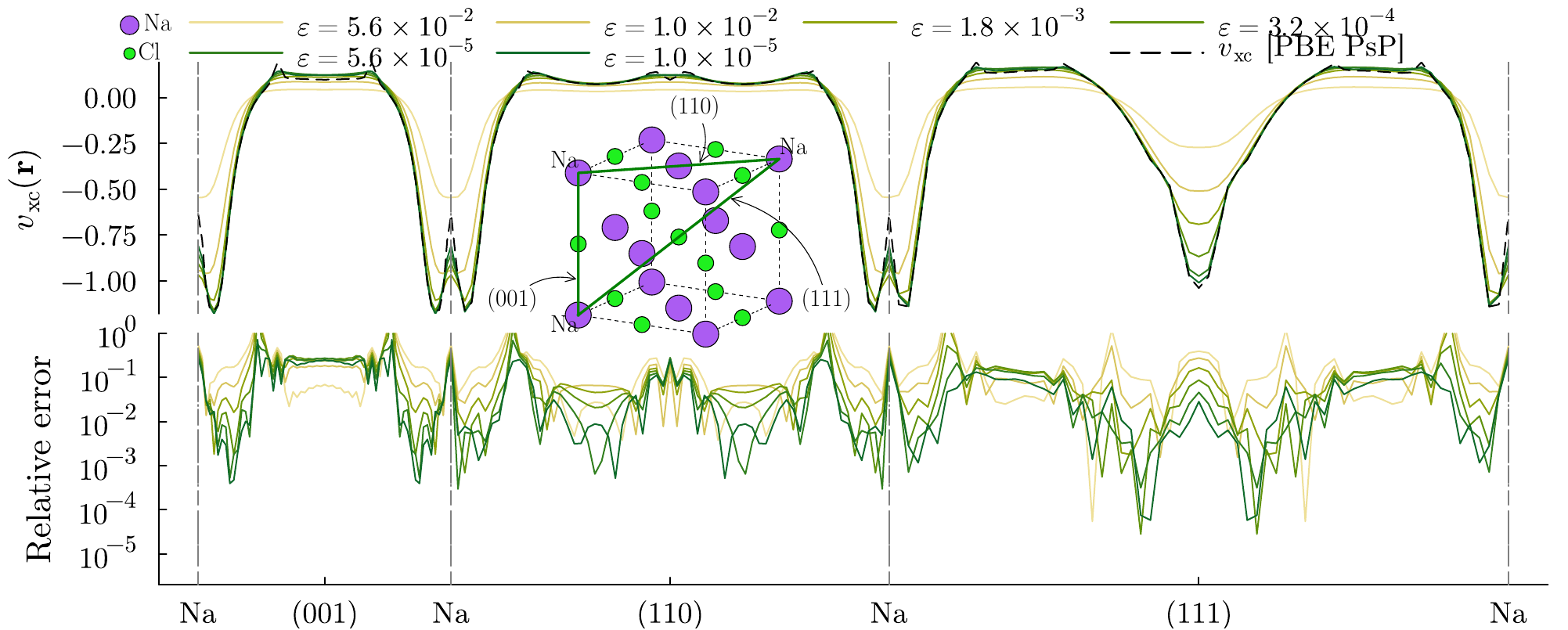}
    \caption{(top) Real-space plot of the reference xc potential for sodium chloride (NaCl) at the level of PBE~\cite{Perdew1996} with pseudopotentials plotted against the potential obtained from inversions for different values of the regularisation parameter $\eps$. The crystalline structure and the plotted path are shown displayed as an inset. (bottom) The corresponding pointwise relative error of the potentials obtained from inversions at various $\eps$ compared to the reference xc potential.}
    \label{fig:NaClPotential}
\end{figure}

The inversion displays a qualitatively similar behaviour as presented in \cite{Herbst_2025} and it accurately recovers the reference xc potential, as illustrated in \cref{fig:NaClPotential}. In particular, \cref{fig:NaClPotential} (top) displays a real-space plot of the potential obtained by the inversion along a closed path in the unit cell, together with the reference xc potential at the level of PBE~\cite{Perdew1996} functional with pseudopotentials. The potential is plotted along the path shown in the inset of \cref{fig:NaClPotential}. For Si, GaAs, and KCl, the analogous figures can be found in \cite{Herbst_2025}, where the inversion was performed in $H^{-1}_\mathrm{per}(\Omega)$ rather than $\Xdensaff$. However, this change of spaces does not qualitatively alter the results. The bottom panel of \cref{fig:NaClPotential} shows the corresponding pointwise relative error. We note that whenever the potential crosses zero, the relative error exhibits an artificial bump. This is a numerical artefact caused by the reference approaching zero and does not reflect an actual deterioration in the quality the reconstructed potential.

\begin{figure}
    \centering
    \includegraphics[width=\linewidth]{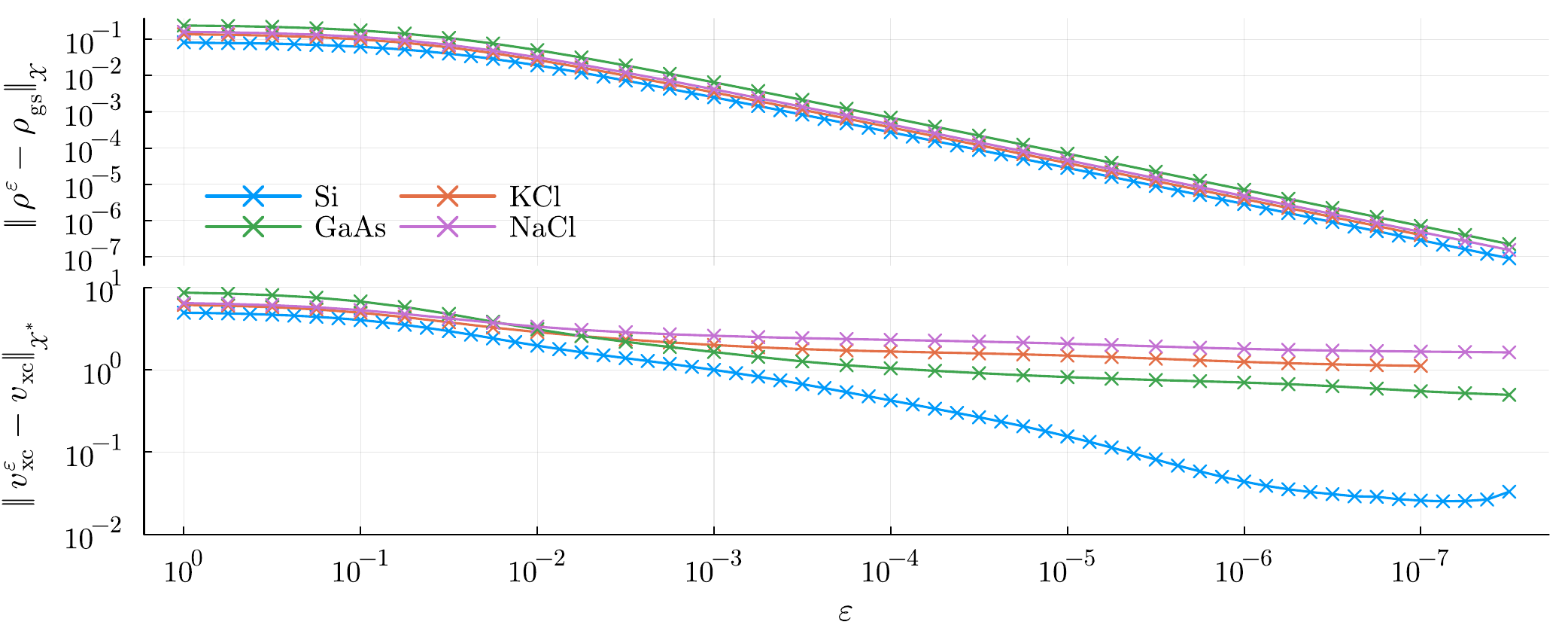}
    \caption{Convergence of KS inversion for Si ($E_\mathrm{cut} = 45$~Ha), KCl ($E_\mathrm{cut} = 65$~Ha), GaAs ($E_\mathrm{cut} = 83$~Ha), and NaCl ($E_\mathrm{cut} = 83$~Ha) with a $k$-point spacing of at most $0.12\,\text{\r{A}}{}^{-1}$ for an exponentially decreasing sequence in $\eps$ using the guiding functional $\mathcal{F}$ (\cref{eq:GuidingFunctional}). (top) The convergence of the proximal densities $\rho^\eps$ to the reference ground-state density $\rho_\mathrm{gs}$. (bottom) The convergence of the determined $v_\mathrm{xc}^\eps$ to the reference $v_\mathrm{xc}$ at the level of PBE~\cite{Perdew1996} with pseudopotentials.}
    \label{fig:Convergence}
\end{figure}

Furthermore, we investigate the convergence properties of the inversion scheme. For all four bulk materials, the proximal density is observed to converge towards the reference ground-state densities $\rho_\mathrm{gs}$, as shown in the upper panel of \cref{fig:Convergence}. For $\eps \lesssim 10^{-2}$, the proximal densities exhibits linear convergence in norm for all systems. This behaviour is also observed for the GPE  (\cref{fig:GPProximalConvergence}), where the onset of the linear regime occurs at $\eps
\simeq 10^{-1}$. The lower panel of \cref{fig:Convergence} shows the convergence of $v_\mathrm{xc}^\eps$ to $v_\mathrm{xc}$ as $\eps$ decreases. While the proximal density converges reliably and largely independently of the material, the convergence of the potential depends more strongly on the systems. Taken together, \cref{fig:GPProximalConvergence,fig:Convergence} indicate a clear relationship between the smoothness of the target potential and both the norm error and the convergence rate of the reconstructed potential. As illustrated by the corresponding real-space plots, resolving the sharpest features demands increasingly small values of $\eps$ as seen near the sodium atoms in \cref{fig:NaClPotential}.

\begin{figure}
    \centering
    \includegraphics[width=\linewidth]{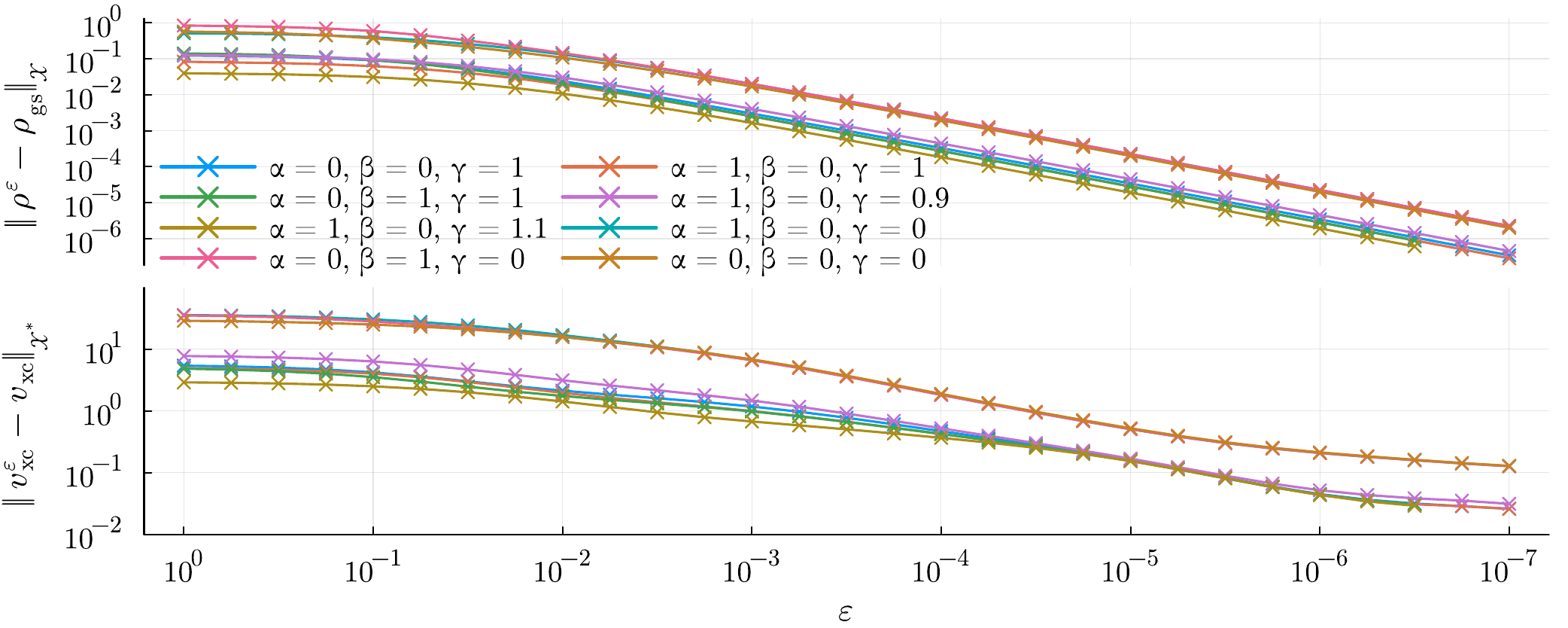}
    \caption{Convergence of KS inversion for Si using different guiding functionals, $\mathcal{F}_{\alpha,\beta,\gamma}$ (\cref{eq:variableGuideFunc}) with $E_\mathrm{cut} = 45$~Ha and a $k$-point spacing of at most $0.12\, \text{\r{A}}{}^{-1}$. Here $\alpha=1$, $\beta=0$, $\gamma=1$ corresponds to the inversion of \cref{fig:Convergence}. (top) The convergence of the proximal densities $\rho^\eps = \Pi^\varepsilon_{\mathcal{F}_{\alpha,\beta,\gamma}} (\rho_\mathrm{gs})$ to the reference ground-state density $\rho_\mathrm{gs}$. (bottom) The convergence of the determined $v_\mathrm{xc}^\eps$ to the reference $v_\mathrm{xc}$, where for all guiding functionals, the determined potentials are appropriately shifted to compare the exchange-correlation potentials only.}
    \label{fig:FunctionalsConvergence}
\end{figure}

Next, we investigate the influence of the Hartree energy $E_\mathrm{H}(\rho)$, the exact Hartree potential $v_\mathrm{H}(\rho_\mathrm{gs})$, and the external potential $v_\mathrm{ext}$ on the performance of the inversion. To this end, we utilise the ($\alpha, \beta, \gamma$)-dependent guiding functional $\mathcal{F}_{\alpha,\beta,\gamma}$ defined in \cref{eq:variableGuideFunc}. Focusing on bulk silicon, the inversion scheme is applied across a range of parameters combinations
$\alpha, \beta ,\gamma \geq 0$, and the resulting convergence behaviour is shown in \cref{fig:FunctionalsConvergence}. The upper panel displays the convergence of the proximal density $\rho^\eps = \Pi^\varepsilon_{\mathcal{F}_{\alpha,\beta,\gamma}}(\rho_\mathrm{gs})$ to $\rho_\mathrm{gs}$, while the lower panel shows the corresponding convergence of the reconstructed potentials. From \cref{fig:FunctionalsConvergence}, we observe that the external potential is the most significant contribution of the guiding functional to the accuracy of the reconstructed potential. Comparing the ``pure'' KS inversion ($\alpha=\beta=\gamma=0$) to the inversion based on $\mathcal{F}$ ($\alpha=\gamma=1$, $\beta=0$), we find that the external potential improves the accuracy by half an order of magnitude at a fixed value of $\eps$. Furthermore, the results in \cref{fig:FunctionalsConvergence}, together with an analogous calculation for potassium chloride, show that over-weighting the external potential by 10\% further improves the convergence of the inversion scheme for $\eps\gtrsim 10^{-4}$.  In contrast, replacing the Hartree energy term with the exact Hartree potential only produces a modest improvement, and only for the largest values of $\eps$.

\begin{figure}
    \centering
    \includegraphics[width=\linewidth]{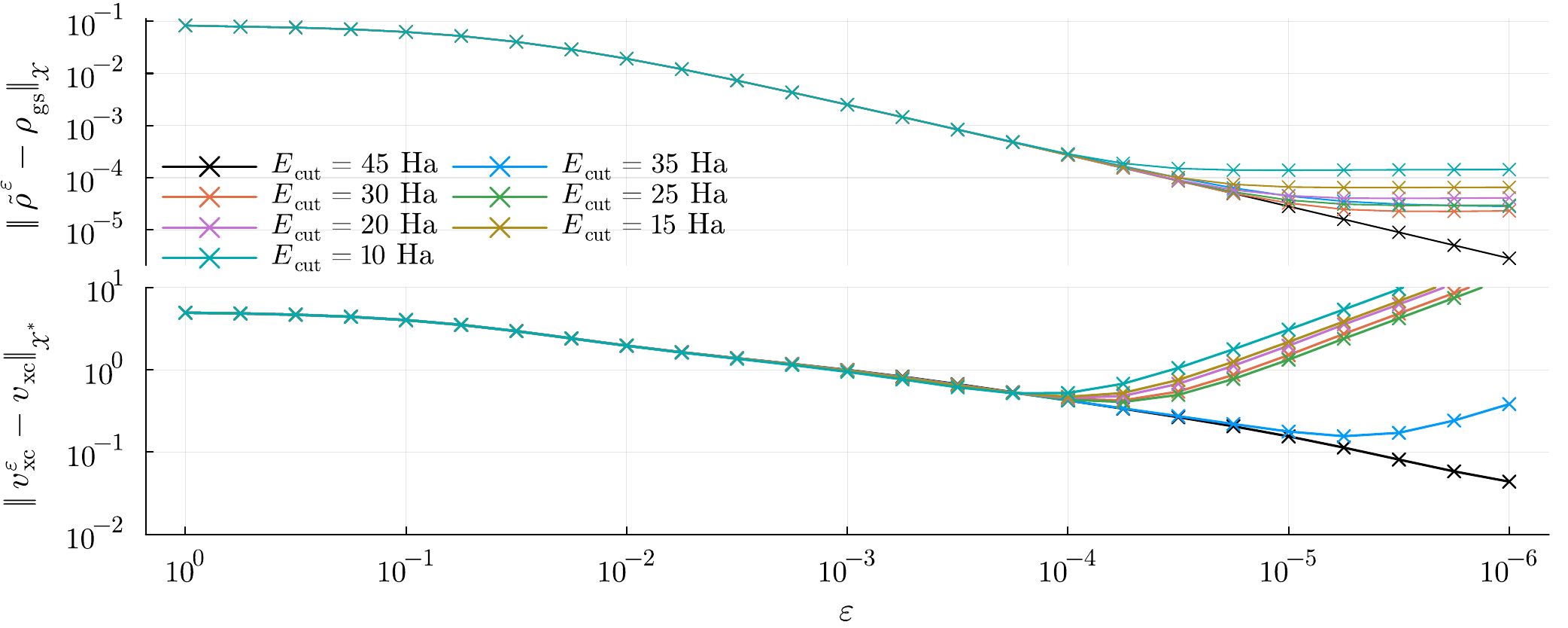}
    \caption{The convergence of the KS inversion with references density $\tilde{\rho}_\mathrm{gs} = \rho_\mathrm{gs} + \Delta \rho$ perturbed by truncations of $\rho_\mathrm{gs}$ in the Fourier representation and the guiding functional $\mathcal{F}$ (\cref{eq:GuidingFunctional}). The $k$-point spacing is at most $0.12\,\text{\r{A}}{}^{-1}$ and $E_\mathrm{cut}=45$~Ha represents the unperturbed reference and $\norm{\Delta \rho}_\Xdens$ is in the range of $2.7 \times 10^{-5}$ to $1.5 \times 10^{-4}$ for the perturbed densities. (top) The convergence of the proximal densities $\tilde{\rho}^\eps = \Pi^\varepsilon_\mathcal{F}(\tilde{\rho}_\mathrm{gs})$ to the exact reference $\rho_\mathrm{gs}$. (bottom) The convergence of the determined $\tilde{v}_\mathrm{xc}^\eps$ to the reference $v_\mathrm{xc}$.}
    \label{fig:SiDensityTruncation}
\end{figure}

Lastly, we introduce perturbations to the reference density to investigate the stability of the inversion scheme. Let $\tilde{\rho}_\mathrm{gs} = \rho_\mathrm{gs} + \Delta \rho$ denote an inexact reference density, where $\rho_\mathrm{gs}$ is the true ground-state density. Similarly to \cite{Herbst_2025}, we generate perturbations by truncating the Fourier representation of the ground-state density and interpolating the truncated density back onto the original grid. In practice, the size of the truncation is controlled by using a smaller $E_\mathrm{cut}$ than for the reference calculation. Using the guiding functional $\mathcal{F}$ (\cref{eq:GuidingFunctional}), we apply the inversion to bulk silicon, where the reference density is computed with $E_\mathrm{cut}=45$~Ha and the truncations introduced using cut-off values between 10~Ha to 35~Ha. This yields perturbations $\Delta\rho$ whose $\Xdens$-norm range from $1.5 \times 10^{-4}$ to  $2.7 \times 10^{-5}$. The convergence of the proximal densities $\tilde{\rho}^\eps = \Pi^\varepsilon_{\mathcal{F}}(\tilde{\rho})$ to the unperturbed reference $\rho_\mathrm{gs}$, as well as the corresponding potentials determined by the inversion, are shown in \cref{fig:SiDensityTruncation}. From the upper panel of \cref{fig:SiDensityTruncation} we observe the maximal achievable accuracy of the proximal density is limited by the magnitude of the imposed truncation, and analogous behaviour is seen for the potentials in the lower panel. Additionally, we observe that the divergence from references occurs when the size of the truncation becomes comparable to the regularisation parameter.

\subsection*{Acknowledgements} 
AL, MP, and VHB were supported by ERC-2021-STG grant agreement No.~101041487 REGAL. AL was also supported by Research Council of Norway through funding of the CoE Hylleraas Centre for Quantum Molecular Sciences Grant No.~262695. MFH acknowledges support by the Swiss National Science Foundation (SNSF, Grant Nos.~221186 and 10002757) as well as the NCCR MARVEL, a National Centre of Competence in Research, funded by the SNSF (Grant No.~205602). MP further acknowledges support from the German Research Foundation (Grant SCHI 1476/1-1). The authors thank Benedikt Menges for useful discussions.

\subsection*{Data Availability Statement}
The data and code supporting the findings of this study are openly available in a public repository on GitHub at: \href{https://github.com/vebjorhb/MY-periodic-inversion}{https://github.com/vebjorhb/MY-periodic-inversion}. The repository contains all datasets and implementation details necessary to reproduce the numerical results presented in this article.

\section*{References}
\bibliography{refs}

\end{document}